\documentclass[a4paper,10pt]{amsart}

\usepackage{amssymb,amsmath,amsfonts}
\usepackage[alphabetic]{amsrefs}
\usepackage[mathscr]{eucal}
\usepackage{graphicx}

\usepackage[all]{xy}

\CompileMatrices

\theoremstyle{plain}
\newtheorem{theorem}{Theorem}[section]

\newtheorem{proposition}[theorem]{Proposition}

\theoremstyle{definition}
\newtheorem{definition}[theorem]{Definition}

\theoremstyle{remark}
\newtheorem{example}[theorem]{Example}
\newtheorem{remark}[theorem]{Remark}

\newtheorem{notation}[theorem]{Notation}

\numberwithin{equation}{section}

\DeclareMathOperator{\Lip}{Lip}
\DeclareMathOperator{\dist}{dist}
\DeclareMathOperator{\expected}{\mathbb{E}}
\DeclareMathOperator{\prob}{Prob}
\DeclareMathOperator{\spt}{spt}

\begin{document}

\title[Fractals with Partial Self Similarity]{$ V $-Variable Fractals: \\ Fractals with Partial Self Similarity}  

\author{Michael  Barnsley}
\address[Michael  Barnsley  and John~E. Hutchinson]
{Department of Mathematics\\
Mathematical Sciences Institute\\
Australian National University\\
Canberra, ACT, 0200\\
AUSTRALIA}
\email[Michael  Barnsley]{Michael.Barnsley@maths.anu.edu.au}
\author{John~E. Hutchinson}
\email[John~E. Hutchinson]{John.Hutchinson@anu.edu.au} 
\author{ \"{O}rjan Stenflo}
\address[\"{O}rjan Stenflo]
{Department of Mathematics\\
Uppsala University\\
751 05 Uppsala\\
SWEDEN}
\email{stenflo@math.uu.se}

\begin{abstract}  We establish properties of a new type of fractal which has partial self similarity at all scales.  For any collection  of iterated functions systems with an associated probability distribution and any positive integer $ V $ there is a corresponding class of $ V $-variable fractal sets or measures with a natural probability distribution. These $ V $-variable fractals can be obtained from the points on the attractor of a single deterministic iterated function system. Existence, uniqueness and approximation results are established under average contractive assumptions.  We also obtain extensions of some  basic results  concerning iterated function systems.
\end{abstract}

\maketitle

\section{Introduction}  \label{secint}

A $ V $-variable fractal is loosely characterised by the fact that it possesses at most $ V $ distinct local patterns at each level of magnification,  where the class of patterns depends on the  level, see Remark~\ref{rm32}.   Such fractals are  useful for modelling purposes and for  geometric applications which require  random fractals with  a controlled degree of strict self similarity at each scale, see  \cite{Bar}*{Chapter 5}.

Standard fractal sets or measures determined by a single iterated function system [IFS] $ F $ acting on a metric space $ X $ such as $ \mathbb{R}^k $, can be generated directly by a deterministic process,      or alternatively by a Markov chain or ``chaos game'', acting on   $ X $.  Now let $ \boldsymbol{F} $ be a family of IFSs acting on $ X $ together with an associated probability distribution   on   $ \boldsymbol{F} $.  Let $ V $ be any positive integer.
The corresponding class of $ V $-variable fractal sets or measures  from $X $, and its associated probability distribution, can be generated by a Markov chain or ``chaos game''  operating \emph{not} on the state space $ X $   but on the state space $ \mathcal{C} (X )^V $ or $  \mathcal{M} (X )^V $ of $ V $-tuples of compact subsets of $ X $ or probability  measures over $X $, respectively.  See Theorems~\ref{th41} and~\ref{th42}, and see Section~\ref{2vf} for a simple example.  The Markov chain converges exponentially, and approximations to its steady state attractor can readily be obtained.  The projection of the attractor in any of the $ V $ coordinate directions  gives the class of $ V $-variable fractal  sets or measures corresponding to $ \boldsymbol{F} $  together  with its natural probability distribution in each case.
The full attractor contains further information about the correlation  structure  of subclasses of these $ V $-variable fractals.

\begin{figure}[h]
\scalebox{.6}{\includegraphics{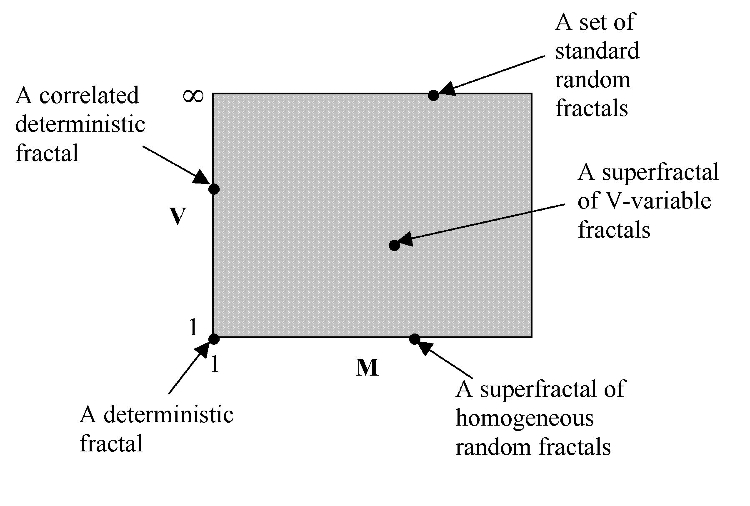}}
\caption{Sets of $ V $-variable fractals for different $ V $ and $ M $. See Remark~\ref{hft}.} 
\label{Table}
\end{figure}

The case $ V=1 $ corresponds to homogeneous random fractals. The limit $ V\to \infty $ gives standard random fractals, and  for this reason the Markov chain for large $ V $ provides a fast way of generating classes of standard random fractals together with their probability distributions.  Ordinary fractals generated by a single IFS can be seen as special cases of the present construction and this provides new insight into the  structure of such fractals, see Remark~\ref{hft}.  
For the connection  with other classes of fractals in the literature see Remarks~\ref{compare} and~\ref{gdf}. 

\medskip  We summarise  the main notation and results.  

Let $ (X,d ) $ be a complete separable metric space.  Typically  this  will be  Euclidean space $  \mathbb{R}^k  $ with the standard metric.  For each $ \lambda $ in some index set $ \Lambda $  let    
$  F^\lambda $ be an IFS  acting on  $ (X,d ) $, i.e.
\begin{equation} \label{dfF1}
F^\lambda= (f^\lambda_1,\dots,f^\lambda_M,w^\lambda_1,\dots,w^\lambda_M ), 
\quad f_m^\lambda : X \to X, 
\quad  0\leq w_m^\lambda \leq 1, \quad  \sum_{m=1}^M w^\lambda _m = 1 . 
\end{equation} 
We will require both the  cases  where $ \Lambda $ is  finite and   where $ \Lambda $ is  infinite.  In order to simplify the exposition we assume that there is only a finite number  $ M $   of functions in each $ F^ \lambda $ and that $ M $  does not depend on~$ \lambda $.  
Let $ P $ be a probability distribution  on some $ \sigma $-algebra of subsets of $ \Lambda $. The given data is then denoted by 
 \begin{equation} \label{dfF2}
\boldsymbol{F} = \{(X,d), F^ \lambda , \lambda \in \Lambda, P \}.
\end{equation} 

Let $ V $ be a fixed positive integer.

\medskip
 Suppose first that the family $\Lambda $ is finite and the functions $ f^\lambda_m $ are uniformly contractive.  The  set $ \mathcal{K}_V$ of \emph{$ V $-variable} fractal subsets of $ X $ and the set $ \mathcal{M}_V$ of \emph{$ V $-variable} fractal measures  on $ X $ associated to $ \boldsymbol{F} $ is then given by Definition~\ref{Vvar}.    
There are   Markov chains   acting  on the set $ \mathcal{C}(X)^V $  of $ V $-tuples of compact subsets of $  X $  and on the set $ \mathcal{M}_c(X)^V $  of $ V $-tuples of compactly supported unit mass measures on $ X $,  whose stationary distributions  project  in any of the $ V $ coordinate directions  to probability distributions  $ \mathfrak{K}_V $ on  $ \mathcal{K}_V $ and  $ \mathfrak{M}_V $   on  $ \mathcal{M}_V $, respectively.  
Moreover, these Markov chains are each given by a  single deterministic  IFS  $ \mathfrak{F}_V^{\mathcal{C} }$     or  $ \mathfrak{F}_V^{\mathcal{M}_c }$  constructed from $ \boldsymbol{F} $ and acting on  $ \mathcal{C}(X)^V $  or $ \mathcal{M}_c(X)^V $  respectively.  
The IFS's $ \mathfrak{F}_V^{\mathcal{C} }$     and  $ \mathfrak{F}_V^{\mathcal{M}_c }$    are called   \emph{superIFS}'s.  
The  sets $ \mathcal{K}_V $ and $ \mathcal{M}_V $, and the probability distributions   $ \mathfrak{K}_V $  and  $ \mathfrak{M}_V $, are called \emph{superfractals}.    
See Theorem~\ref{th41}; some of these results were first obtained in~\cite{BHS05}. The distributions $ \mathfrak{K}_V $   and  $ \mathfrak{M}_V $ have a complicated correlation structure and differ markedly  from other notions of random fractal in the literature.  See Remarks~\ref{compare} and~\ref{hcf}.

In many situations one needs an infinite family $ \Lambda $ or  needs average contractive conditions, see Example~\ref{rm54}. 
In this case one works with the set $ \mathcal{M}_1(X)^V $  of $ V $-tuples of finite first moment  unit mass measures on $ X $. The corresponding  superIFS  $ \mathfrak{F}_V^{\mathcal{M}_1 }$ is pointwise average contractive by Theorem~\ref{th42} and one obtains the existence of   a corresponding superfractal distribution 
$ \mathfrak{M}_V $. There are  technical difficulties in establishing these results, see Remarks~\ref{rmloc}, \ref{nosep}  and~\ref{npc}.  

\medskip 

In Section~\ref{xd} the properties of the   the Monge-Kantorovitch and the strong Prokhorov probability metrics  are summarised.   The strong Prokhorov metric is not widely known in the fractal literature although it is the natural metric to use with uniformly contractive conditions. We include the mass transportation, or equivalently the probabilistic,  versions of  these metrics, as  we need them  in Theorems~\ref{th1}, \ref{th41}, \ref{th42} and~\ref{rc}.   
We  work with probability metrics on spaces of measures and such metrics are not always separable.  So in Section~\ref{xd} the extensions required to include non separable spaces are noted. 
  
In  Section~\ref{secdf}, particularly Theorem~\ref{th1} and the following Remarks,  
we summarise  and in some cases extend  basic results in the literature concerning IFS's, link the measure theoretic and probabilistic approaches and sketch the proofs.  In  particular, IFS's with a possibly infinite family of functions and pointwise average contractive conditions are considered. 
The law of large numbers for the corresponding Markov process starting from an arbitrary point, also known as the convergence theorem for the ``chaos game'' algorithm, is extended to the case when the IFS acts on a non
locally compact state space.
This situation typically occurs when the state space is a function space or space of measures, and   here it is required  for the superIFS $ \mathfrak{F}_V^{\mathcal{M}_1 }$ in Theorem~\ref{th42}.  The strong Prokhorov metric is used 
in the case of uniform contractions.    We hope Theorem~\ref{th1}  will be of independent use. 

In Section~\ref{secifsrf}  we summarise some of the basic properties  of  standard random fractals generated by a family of IFS's.

In Section~\ref{secvvtc}   the representation of $ V $-variable fractals  in  the space $ \Omega_V $ of tree codes and in the space $ \mathcal{A}_V^\infty $ of addresses  is developed, and  the connection between the two spaces is discussed.  The space $ \Omega_V $ is of the type  used for realisations of general random fractals  and consists of trees with each node  labelled in effect by an IFS, see the comment following Definition~\ref  {dftrees} and see Definition~\ref{Vvar}.  The space $ \mathcal{A}_V^\infty $   is of the type used to address points on a single deterministic fractal, and here consists of infinite sequences of $ V \times (M+1) $ matrices  each of which  defines a map from the set of $ V $-tuples of sets or measures to itself, see Definition~\ref{dfseqcode}.  Also see Figures~\ref{fig:example} and~\ref{fig:example2}.   

In Section~\ref{seccgvvf} the existence, uniqueness and convergence results for $ V $-variable fractals and superfractals are proved, some examples are given, and the connection with graph directed IFSs is discussed. In Section~\ref{secvti} we establish the rate at which the probability distributions $ \mathfrak{K}_V $ and $ \mathfrak{M}_V $ converge to the corresponding distributions on standard random fractal sets and measures respectively as $ V \to \infty $.
In Section~\ref{2vf}   a simple example of a super IFS and the  associated Markov chain is given. 
 
 In Section~\ref{secext}  we make some concluding remarks including the relationship with other types of fractals,  extensions of the results, and some motivation for the method of construction of $ V $-variable fractals. 
 
  \medskip The reader may find it easier to begin with Section~\ref{secifsrf} and refer back to Sections~\ref{secbac}  and~\ref{secdf} as needed, particularly in the proofs of Theorems~\ref{th41} and~\ref{th42}. 
  An index of notation is provided at the end of the paper.

\medskip This work was partially supported by the Australian Research Council and carried out at the Australian National University.


\section{Preliminaries}\label{secbac}\label{xd} 

Throughout the paper $ (X , d)$ denotes a complete separable metric space, except where mentioned otherwise.  

\begin{definition}\label{cx}
 The collection of  nonempty compact subsets of $ X$ with the     Hausdorff metric is denoted by $ (\mathcal{C} (X), d_{\mathcal{H}}) $. The collection of  nonempty bounded closed subsets of $ X$ with the      Hausdorff metric is denoted by  $ (\mathcal{B} \mathcal{C} (X), d_{\mathcal{H}}) $. 

For $  A \subset X $ let
$ 
A^\epsilon =  \{x: d(x  ,A) \leq \epsilon \} 
$ 
 be the \emph{closed  $ \epsilon $-neighbourhood} of $ A  $.
\end{definition} 

Both spaces $ (\mathcal{C} (X), d_{\mathcal{H}}) $ and $ (\mathcal{B} \mathcal{C} (X), d_{\mathcal{H}}) $  are complete and separable if $ (X, d) $ is complete and separable. 
Both spaces are complete  if $ (X, d) $ is just assumed to be complete. 

\begin{definition}[\emph{Prokhorov metric}]\label{mx}
 The collection of  unit mass Borel (i.e.\  probability) measures on the Borel subsets of $ X $ with the topology of weak convergence is denoted by $ \mathcal{M}(X) $. Weak convergence of $ \nu_n \to \nu $ means   $ \int  \phi \, d\nu_n \to \int  \phi \, d\nu $ for all bounded continuous  $ \phi $.  The (standard) \emph{Prokhorov metric} $ \rho $ on  $ \mathcal{M}(X) $ is defined by
\[
\rho(\mu,\nu) := \inf \left\{ \epsilon > 0 : \mu(A)\leq \nu(A^\epsilon) + \epsilon \text{ for all Borel sets $ A \subset X $ }\right\} .
\]
\end{definition}

The Prokhorov metric  $ \rho $  is complete and separable and induces the topology of weak convergence.   Moreover,   $ \rho $  is complete if $ (X, d) $ is just assumed to be complete.  See \cite{Bil}*{pp 72,73}.  We will not use the Prokhorov metric but mention it for comparison with the strong Prokhorov metric in Definition~\ref{mp}.

The   \emph{Dirac measure} $ \delta_a $ concentrated at  $  a $ is defined by $ \delta_a (E) = 1$
if $ a\in E $ and otherwise   $ \delta_a (E) = 0$.
 
\medskip If $ f:X \to X $ or $ f:X \to \mathbb{R} $,  the \emph{Lipschitz constant} for $ f$  is denoted by $ \Lip  f  $  and   is  defined to be the least   $ L $ such that $ d(f(x), f(y) ) \leq L d(x,y)$   for all $ x,y\in X $.

\begin{definition}[\emph{Monge-Kantorovitch  metric}]\label{mx2} The collection of those  $ \mu \in \mathcal{M}(X) $  with finite first moment,  i.e.\ those $ \mu $ such that 
 \begin{equation} \label{dfmk}
  \int d(a,x) \, d\mu (x) <\infty
 \end{equation} 
  for some and hence any $ a\in X $, is denoted by $  \mathcal{M}_1(X)$.
The  \emph{Monge-Kantorovitch  metric}  $ d_{MK} $ on $  \mathcal{M}_1(X)$
is defined in any of the three following equivalent ways:
 \begin{equation} \label{mk}
\begin{aligned}
&d_{MK}(\mu,\mu')   :=  \sup_f \bigg\{ \int f\, d\mu - \int f \, d\mu' : \text{Lip\,} f\leq 1\bigg\}   \\ 
 &\qquad  =  \inf_\gamma \bigg\{ \int d(x,y)\, d\gamma(x,y) : \gamma \text{ a Borel measure on } X\times X,\, \pi_1(\gamma)=\mu, \, \pi_2(\gamma) = \mu' \bigg\}  \\
& \qquad  =  \inf \bigg\{ \expected d(W,W') : \dist W = \mu,\, \dist W' = \mu' \bigg\}.
  \end{aligned} 
  \end{equation} 
The maps $  \pi_1, \pi_2:X\times X \to X $ are the projections onto the first and second coordinates, and so $ \mu $ and $ \mu'$ are the marginals  for $ \gamma $.    
 In the third version the infimum is taken over $ X $-valued random variables $ W  $ and $ W'$ with distribution $ \mu  $ and $ \mu' $ respectively but otherwise unspecified joint distribution.
  \end{definition}

Here and elsewhere, ``dist'' denotes the probability distribution  on the associated random variable.

The metric space $ ( \mathcal{M}_1(X), d_{MK} ) $ is complete and separable.   
The moment restriction \eqref{dfmk}  is automatically satisfied if  $ (X, d) $ is bounded.  The second equality in \eqref{mk} requires proof, see  \cite{Dud}*{\S11.8}, but the third form of the definition is just a rewording of the second.
   The connection between $ d_{MK} $ convergence in $ \mathcal{M}_1(X) $ and   weak convergence is given by
  \[
  \nu_n \xrightarrow{d_{MK}} \nu  \quad \text{iff} \quad 
   \nu_n \xrightarrow{} \nu \text{ weakly and } \int d(x,a)\, d\nu_n(x) \to \int d(x,a)\, d\nu(x)
   \]
   for some and hence any $ a\in X $.   See \cite{Vil}*{Section 7.2}.
   
  Suppose  $ (X,d) $ is only assumed to be complete.    If measures $ \mu $ in $ \mathcal{M}_1(X) $ are also required to satisfy the condition $ \mu(X\setminus \spt\mu) = 0 $, then  $ ( \mathcal{M}_1(X), d_{MK} ) $ 
is complete, see~\cite{Kra} and  \cite{Fed}*{\S2.1.16}. This condition  is satisfied  for all finite Borel measures $ \mu $ if $ X $ has a dense subset whose cardinality is an Ulam number, and in particular if $ (X,d )$ is separable.  The requirement that the cardinality of $ X $ be an Ulam number  is not   very restrictive, see~\cite{Fed}*{\S2.2.16}.

\medskip

It is often more natural to use the following  strong Prokhorov metric rather than the Monge-Kantorovich or standard Prokhorov metrics in the case of  a uniformly contractive IFS.

\begin{definition}[\emph{Strong Prokhorov metric}]\label{mp}
The set  of compact  support, or bounded support,  unit mass Borel  measures  on $ X $  is denoted  by   $ \mathcal{M}_c (X)$ or $  \mathcal{M}_b(X) $, respectively. 

The     \emph{strong Prokhorov metric} $ d_P $ is defined on  $ \mathcal{M}_b(X) $
 in any of the following equivalent ways:
\begin{equation} \label{p}
\begin{aligned}
d_P(\mu,\mu') :=& \inf \left\{ \epsilon > 0 : \mu(A)\leq \mu'(A^\epsilon) \text{ for all Borel sets $ A \subset X $ }\right\}   \\
=& \inf \left\{ \text{ess sup}_\gamma\, d( x, y) : 
\gamma \text{ is a measure on } X\times X,\, \pi_1(\gamma)=\mu, \, \pi_2(\gamma) = \mu' \right\}  \\
=&  \inf \left\{ \text{ess sup}\, d( W, W') : 
            \dist W = \mu,\, \dist W' = \mu' \right\}, 
            \end{aligned} 
\end{equation}
where the notation is as in the paragraph following \eqref{mk}.  
\end{definition}
Note that
\[
 \mathcal{M}_c (X) \subset  \mathcal{M}_b(X)   \subset \mathcal{M}_1 (X) \subset \mathcal{M}(X).
\]
The first definition in \eqref{p} is symmetric in $ \mu $ and $  \mu'  $  by a standard argument, see~\cite{Dud}*{proof of Theorem 11.3.1}.  For discussion and proof of the second equality see \cite{Rachev}*{p160 eqn(7.4.15)}  and the other references mentioned there.  The third version is a probabilistic reformulation of the second.

\begin{proposition}\label{prop25}
$(\mathcal{M}_c(X), d_P ) $ and $(\mathcal{M}_b(X), d_P ) $  are complete.
If $ \nu,\nu' \in \mathcal{M}_b(X) $ then
\begin{equation*} \label{}
d_{\mathcal{H} } (\text{spt\,} \nu, \text{spt\,} \nu')  \leq d_P(\nu,\nu'), \quad 
d_{MK}(\nu,\nu')  \leq d_P(\nu,\nu').
\end{equation*} 
In particular,  $ \nu_k\to\nu $ in the $ d_P $ metric implies $ \nu_k\to\nu $ in the $ d_{MK} $ metric and $ \text{spt\,} \nu_k \to  \text{spt\,} \nu $ in the $ d_{\mathcal{H}} $ metric.  
\end{proposition}

\begin{proof}
The first inequality follows  directly from the definition of the Hausdorff metric and the second from  the final characterisations in \eqref{mk} and \eqref{p}.

 Completeness  of   $ \mathcal{M}_b(X) $ can be shown as follows  and this argument carries across to $ \mathcal{M}_c(X) $. Completeness of $ \mathcal{M}_c(X) $    is also shown in  \cite{Falconer86}*{Theorem 9.1}.

Suppose $ (\nu_k)_{k\geq 1} \subseteq (\mathcal{M}_b(X), d_P )$ is   $ d_P $-Cauchy.   It follows that $ (\nu_k)_{k\geq 1} $ is   $ d_{MK} $-Cauchy and hence converges to some measure $ \nu $   in the $ d_{MK} $ sense and in particular weakly.  Moreover, $ \text{spt\,}(\nu_k)_{k\geq 1} $ converges to some bounded closed set $ K $ in the Hausdorff sense,  hence $ \text{spt\,} \nu \subset K $ using weak convergence, and so $ \nu \in   \mathcal{M}_b(X) $.  Suppose $ \epsilon > 0 $ and using the fact $ (\nu_k)_{k\geq 1} $ is  $ d_P $-Cauchy choose $ J $  so $ k,j \geq J $ implies $ \nu_k(A) \leq \nu_j(A^\epsilon) $ for all Borel $ A\subset X $. By weak convergence and because $ A^\epsilon $ is closed,  $ \limsup_{j\to \infty}\nu_j(A^\epsilon) \leq \nu(A^\epsilon)  $  and  so $ \nu_k(A)\leq \nu(A^\epsilon) $ if $ k \geq J $.  Hence $ \nu_k \to \nu $ in the $ d_P $ sense.
\end{proof}

If $ (X,d) $ is only assumed to be complete,  but measures in $ \mathcal{M}_c(X)  $ and $ \mathcal{M}_b(X) $
  are also required to satisfy the condition $ \mu(X\setminus \spt\mu) = 0 $ as discussed following Definition~\ref{mx2}, then the same proof shows that 
 Proposition~\ref{prop25} is still valid.  The main point is that one still has completeness of 
$ ( \mathcal{M}_1(X), d_{MK} ) $.

\begin{remark}[\emph{The strong and the standard Prokhorov metrics}]\label{nsp} Convergence  in the  strong Prokhorov metric is  a  much  stronger requirement than convergence in the standard Prokhorov metric or the Monge-Kantorovitch metric.  A simple example  is given by $ X=[0,1] $ and 
$ \nu_n = (1-\frac{1}{n}) \delta_0 + \frac{1}{n}\delta_1$.  Then $ \nu_n \to \delta_0$ weakly and in the $ d_{MK} $ and  $ \rho $ metrics, but $ d_P(\nu_n, \delta_0)=1  $ for all~$ n $.     

The  strong Prokhorov metric   is normally not  separable.  For example, if  $ \mu_x = x \delta_0 + (1-x) \delta_1 $ for $ 0<x< 1 $ then $d _P(\mu_x,\mu_y)=1 $   for $ x\neq y $.     So there is no countable dense subset.
\end{remark}

If $ f:X \to X $ is Borel measurable then the \emph{pushforward  measure} $  f(\nu)  $ is defined by $  f(\nu)(A) = \nu(f^{-1}(A))   $ for Borel sets $ A $. 
 The \emph{scaling property} for Lipschitz functions $ f $, namely
\begin{equation} \label{spl}
 d_P(f(\mu),f(\nu)) \leq \Lip f\   d_P(\mu,\nu),
\end{equation}
follows from the definition of $ d_P $.   
 Similar properties are well known and easily established for the Hausdorff and Monge-Kantorovitch metrics.

\section{Iterated Function Systems}\label{secdf}

\begin{definition}\label{def22} 
An \emph{iterated functions system} [IFS]  $F = (X, f_\theta,  \theta \in \Theta, W) $   is a set of maps $ f_\theta:X\to X $ for $ \theta \in \Theta $,  where $ (X , d) $  is a \emph{complete separable} metric space   and $ W $ is a  probability  measure      on some $ \sigma $-algebra of  subsets of $ \Theta $.  The map 
$ (x,\theta) \mapsto f_\theta(x):X\times \Theta \to X $ is   measurable with respect to the product $ \sigma $-algebra on $ X\times \Theta $, using the Borel $ \sigma $-algebra   on $ X $.  
If  $ \Theta = \{1,\dots,M\}   $ is finite and $ W(m) = w_m $  then one writes $ F = (X,f_1,\dots, f_M, w_1,\dots, w_M) $. 
\end{definition}

It follows $ f $ is measurable in $ \theta $ for fixed $ x $ and in $ x $ for fixed $ \theta $.  Notation such as $  \expected_\theta $   is used to denote  taking the expectation, i.e.\ integrating, over the variable $ \theta $ with respect to $ W $.

Sometimes we will need to work with an IFS on a nonseparable metric space.  The properties which still hold in this case will be noted explicitly.  See Remark~\ref{nosep} and Theorem~\ref{th41}. 

The IFS $ F $ acts on subsets of $  X  $ and Borel measures over $ X $, for  finite and infinite  $ \Theta $  respectively, by
\begin{equation} \label{eq21} 
\begin{alignedat}{2}  
F(E)  &= \bigcup_{m=1}^M f_m(E),&\quad 
F(\nu) &= \sum_{m=1}^M w_m f_m(\nu),\\
F(E) &=  \bigcup_\theta f_\theta(E),& \quad 
  F(\nu ) &=   \int \! dW(\theta)\, f_\theta(\nu) .  
  \end{alignedat}
  \end{equation}  
We put aside measurability matters and interpret  the integral formally as   the measure which operates on  any Borel set   $ A\subset X $ to give   
$  \int \! dW(\theta)\,  (f_\theta(\nu))(A)  $.  The latter is   thought of as a weighted sum via $ W(\theta) $ of the measures $ f_\theta(\nu) $.   The precise definition in the cases we need for infinite $ \Theta $ is given by~\eqref{fnu}.

If $ F(E)=E $  or $ F(\nu) = \nu $ then $ E $ or $ \nu $   respectively    is said to be \emph{invariant} under the IFS  $ F $.

\medskip

In the study of fractal geometry one is usually interested in the case of an IFS with a finite  family $ \Theta $  of maps.  If $ X =  \mathbb{R}^2 $ then compact subsets  of $ X $ are often identified with black and white images, while measures are identified with greyscale images.   Images are generated by iterating the map $ F  $ to approximate  $ \lim_{k\to \infty } F^k(E_0) $ or $ \lim_{k\to \infty } F^k(\nu_0) $.  As seen in the following theorem, under natural conditions the limits exist and are independent of the starting set $ E_0 $  or measure~$ \nu_0 $.

In the study of  Markov chains on an arbitrary state space $ X  $ via  iterations of random functions on $ X  $, it is usually more natural to consider the case of an infinite  family $ \Theta $.    One is concerned with a random process $ Z^x_n $, in fact a Markov chain, with initial state $ x \in X $,  and 
\begin{equation} \label{28}
Z^x_0(\boldsymbol{i})=x, \quad  Z^x_{n}(\boldsymbol{i}) := f_{i_n}(Z^x_{n-1}(\boldsymbol{i}))
= f_{i_n}\circ \dots \circ f_{i_1}(x)
\text{ if } n\geq 1,
\end{equation}
where the $ i_n\in \Theta $ are independently and identically distributed [iid]  with probability distribution $  W $ and $ \boldsymbol{i}=i_1i_2 \dots $\, .  The \emph{induced 
probability measure on the set of codes $ \boldsymbol{i}  $}  is also denoted by $ W $.  

Note that the probability $ P(x,B) $ of going from $ x\in X $  into $  B \subset X $ in one iteration is $ W\{\theta: f_\theta(x) \in B \} $,  and $ P(x,\cdot) = \dist   Z^x_1 $.

More generally, if the starting state is given by a random variable $ X_0 $ independent of $ \boldsymbol{i} $ with $ \dist X_0 = \nu $, then one defines the random variable $ Z^\nu_n(  \boldsymbol{i} ) = Z ^{X_0}_n(  \boldsymbol{i} ) $.
The sequence $ \big( Z^\nu_n(  \boldsymbol{i}) \big)_{n\geq 0 } $ forms a Markov chain starting according to $ \nu $.
We define  $ F(\nu) = \dist Z_1^\nu $  and in summary we have
\begin{equation} \label{fnu} 
 \nu := \dist Z_0^\nu, \quad F(\nu) := \dist  Z^\nu_1, \quad F^n(\nu) = \dist  Z^\nu_n.
\end{equation} 

The operator $ F $ can be applied to  bounded  continuous functions $ \phi : X \to  \mathbb{R} $ 
via any of the following equivalent definitions: 
\begin{equation} \label{fphi}
(F(\phi))(x) = \int \phi(f_\theta(x))\, dW(\theta)\ \Bigl(\text{or } \sum_m w_m\phi(f_m(x)) \Bigr) = \expected_\theta\phi (f_\theta(x)) = \expected \phi ( Z^x_1).
\end{equation}
In the context of Markov chains,  the operator $ F $ acting on  functions is called the \emph{transfer operator}.
It follows from the definitions that  
$ \int F(\phi) \, d\mu = \int \phi \, d (F\mu) $,
which is the expected value of $ \phi $ after one time step starting with the initial distribution $ \mu $.  If one assumes 
$ F(\phi )$ is  continuous (it is automatically bounded) then $ F $ acting on measures is the \emph{adjoint} of $ F $ acting on bounded continuous functions.  Such  $  F $ are said to satisfy the \emph{weak Feller property} --- this is the case  if all  $ f_\theta  $ are continuous by the dominated convergence theorem, or if the pointwise average contractive condition is satisfied, see~\cite{Stenflo03-1}.

We will need  to apply the maps in \eqref{28} in the reverse order.  Define
\begin{equation} \label{28new}
 \widehat{Z}^x_0(\boldsymbol{i})=x,  \quad
\widehat{Z}^x_{n}(\boldsymbol{i}) = f_{i_1}\circ \dots \circ f_{i_n}(x)\text{ if } n\geq 1 .
\end{equation}
Then from the iid property of the $ i_n $ it follows 
\begin{equation}\label{nuphi}
F^n(\nu) = \dist  Z^\nu_n = \dist  \widehat{Z}^\nu_n   ,\quad
F^n(\phi)(x) =  \expected \phi(Z^x_n) = \expected \phi(\widehat{Z}^x_n).
\end{equation} 
However, the pathwise behaviour of the processes $  Z^x_{n} $ and $  \widehat{Z}^x_{n} $ are very different. Under  suitable conditions the former is ergodic and the latter is a.s.\ convergent, see  Theorems \ref{th1}.c and \ref{th1}.a respectively, and the discussion in \cite{DF}.

\medskip

The following Theorem \ref{th1} is known with perhaps two exceptions: the lack of a local compactness requirement in (c) and the use of the strong Prokhorov metric in (d).  

The strong Prokhorov metric, first used in the setting of random fractals in \cite{Falconer86}, is a more natural metric than the Monge-Kantorovitch metric  when dealing with uniformly contractive maps and fractal measures in compact spaces.
Either it or variants may be useful in image compression matters.  In Theorem~\ref{th41} we use it to strengthen the convergence results in \cite{BHS05}. 

The pointwise ergodic theorems for Markov chains for \emph{any} starting point as established in \cites{Breiman, Elton87, BEH, Elton90, MT}  require compactness or local compactness, see Remark~\ref{rmloc}. We  remove this restriction in Theorem~\ref{th1}.  The result is needed  in Theorem~\ref{th42} where we   consider an IFS operating on the space $( \mathcal{M}_1(X)^V,d_{MK} )$  of $  V $-tuples of probability measures. Like most spaces of functions or measures,  this space is not locally compact even if $ X= \mathbb{R}^k $.  See Remark~\ref{rmloc} and also Remark~\ref{rm54}.

We assume a pointwise average contractive condition, see Remark~\ref{rmlog}.  We need this in Theorem~\ref{th42}, see Remark~\ref{npc}.

The parts of the theorem have a long history.   In the Markov chain literature the contraction conditions  \eqref{eqmc} and \eqref{equcb}    were introduced in \cite{Isaac} and \cite{FD}
respectively  in order to establish ergodicity.   
 In the fractal geometry literature, following \cites{Man1, Man2}, the existence and uniqueness of attractors, their properties,  and the Markov chain approach to generating fractals, were introduced in \cites{Hutchinson81, BD85, DS, BE, Elton87, Elton90}.  See \cites{Stenflo98, Stenflo03-1, DF} for further developments and more on the history.

\begin{theorem}\label{th1}
Let $F = (X, f_\theta,  \theta \in \Theta, W) $ be an IFS on  a complete separable metric space $ (X,d) $.   Suppose  $ F $ satisfies the pointwise average contractive  and average boundedness  conditions 
\begin{equation} \label{eqmc} 
 \expected_\theta d(f_\theta(x), f_\theta(y))  \leq r d(x,y)
  \quad \text{and} \quad 
L:= \expected_\theta   d(f_\theta(a),a) < \infty 
\end{equation} 
 for some fixed $ 0<r< 1 $, all $ x,y\in X $, and some $ a \in X $. 

\medskip
\noindent \textnormal{\textbf{a.}}  For some function $ \Pi $,  all $ x,y\in X $ and all $ n $, we have
\begin{equation} \label{eqth0} 
 \expected d( Z^x_{n}( \boldsymbol{i}),  Z^y_{n}( \boldsymbol{i}))  
=
\expected  d(\widehat{Z}^x_{n}( \boldsymbol{i}),  \widehat{Z}^y_{n}( \boldsymbol{i})) \leq  r^n d(x,y), 
\quad 
\expected  d(\widehat{Z}^x_{n}( \boldsymbol{i}), \Pi ( \boldsymbol{i})) \leq \gamma_x r^n , 
\end{equation} 
where $ \gamma_x = \expected_\theta d(x,f_\theta(x)) \leq 
 d (x,a)+L/(1-r) $ and $ \Pi( \boldsymbol{i})$ is independent of $ x $. 
The map $ \Pi  $ is called the \emph{address map} from code space into $ X $.  Note that $ \Pi $  is defined only a.e.

If $  r< s<1 $ then for all $ x,y\in X $ for a.e.\ $ \boldsymbol{i}=i_1 i_2 \dots i_n \dots $ there exists $ n_0 = n_0( \boldsymbol{i}, s)$  such that 
 \begin{equation} \label{eqth1}  
 d( Z^x_{n}( \boldsymbol{i}),  Z^y_{n}( \boldsymbol{i})) \leq  s^n d(x,y),
 \quad 
 d(\widehat{Z}^x_{n}( \boldsymbol{i}),  \widehat{Z}^y_{n}( \boldsymbol{i})) \leq  s^n d(x,y), 
\quad 
 d(\widehat{Z}^x_{n}( \boldsymbol{i}), \Pi ( \boldsymbol{i})) \leq  \gamma_x s^n, 
\end{equation} 
for $ n\geq n_0 $.

\medskip
\noindent \textnormal{\textbf{b.}}  If   $ \nu $ is a unit mass Borel measure then
$ 
F^n(\nu) \to \mu \text{ weakly} 
$ 
where $ \mu := \Pi (W)$ is  the projection of the measure $ W $ on code space onto $ X $ via $ \Pi $. Equivalently, $ \mu $ is the distribution of $ \Pi $ regarded as a random variable.   
In particular, 
\[ 
F(\mu) = \mu
\]
and $ \mu $ is the unique  invariant unit mass measure.

The map $ F  $  is a contraction on  $ (\mathcal{M}_1(X), d_{MK}) $ with Lipschitz  constant  $ r  $.  Moreover,  $ \mu \in \mathcal{M}_1(X) $  and 
\begin{equation} \label{mkcon}  
d_{MK} (F^n(\nu), \mu) \leq  r^n d_{MK} ( \nu , \mu) 
\end{equation} 
 for every $ \nu  \in \mathcal{M}_1(X) $.

\medskip
\noindent \textnormal{\textbf{c.}}  For all $  x\in X $ and for a.e.\ $ \boldsymbol{i} $, the empirical measure (probability distribution)
\begin{equation} \label{eqth2} 
\mu^x_n(\boldsymbol{i}):= \frac{1}{n} \sum_{k=1}^n \delta_{ Z^x_k(\boldsymbol{i}) } \to \mu
\end{equation} 
weakly.  Moreover, if $ A  $ is the support of $ \mu  $ then there exists $ n_0 = n_0(\boldsymbol{i},x, \epsilon) $ such that
\begin{equation} \label{eqth3} 
Z^x_n(\boldsymbol{i}) \subset A^\epsilon   \text{ if } n \geq n_0 .
\end{equation} 

\medskip
\noindent \textnormal{\textbf{d.}}  Suppose $  F  $ satisfies the  uniform contractive  and  uniform boundedness  conditions
\begin{equation} \label{equcb} 
\sup\nolimits_\theta d(f_\theta(x), f_\theta(y))  \leq r d(x,y) 
\quad \text {and} \quad 
L:= \sup\nolimits_\theta d(f_\theta(a), a) <\infty 
\end{equation} 
  for  some  $ r< 1  $, all $ x,y\in X $, and  some $ a\in X $. 
   Then 
 \begin{equation} \label{eqth1a}  
 d( Z^x_{n}( \boldsymbol{i}),  Z^y_{n}( \boldsymbol{i})) \leq  r^n d(x,y),
 \quad 
 d(\widehat{Z}^x_{n}( \boldsymbol{i}),  \widehat{Z}^y_{n}( \boldsymbol{i})) \leq  r^n d(x,y), 
\quad 
 d(\widehat{Z}^x_{n}( \boldsymbol{i}), \Pi ( \boldsymbol{i})) \leq  \gamma_x r^n, 
\end{equation} 
for all $ x,y\in X $ and all $ \boldsymbol{i} $.   The address map  $ \Pi $  is everywhere defined  and is continuous with respect to the product topology defined on code space and induced from the  discrete metric on $ \Theta $.
  Moreover,  $ \mu =\Pi(W) \in \mathcal{M}_c $  and for any $ \nu  \in \mathcal{M}_b $, 
\begin{equation} \label{eqth4a}
d_{P} (F^n(\nu), \mu) \leq  r^n d_{P} ( \nu , \mu).
\end{equation} 

Suppose in addition $ \Theta $ is finite  and  $ W(\{\theta\})>0 $ for $ \theta \in \Theta $.  Then $ A  $ is compact  and for any closed bounded $ E $ 
\begin{equation} \label{eqth5a} 
d_{\mathcal{H} }( (F^n(E), A) \leq  r^n d_{\mathcal{H} } (E , A).
\end{equation} 
Moreover, $ F(A) = A $ and  $ A $  is the unique closed bounded invariant set.
\end{theorem}

\begin{proof} 
 
\noindent \textbf{a.} The first inequality in \eqref{eqth0} follows from \eqref{28}, \eqref{28new} and contractivity.

Next fix  $ x $.  Since 
\begin{equation} \label{ineq1} 
\begin{aligned} 
\expected d(\widehat{Z}^x_{n+1}( \boldsymbol{i}), \widehat{Z}^x_{n}( \boldsymbol{i})) &\leq r \expected d(\widehat{Z}^x_{n}( \boldsymbol{i}), \widehat{Z}^x_{n-1}( \boldsymbol{i})),\\ 
 \expected d(\widehat{Z}^a_{n+1}( \boldsymbol{i}), \widehat{Z}^a_{n}( \boldsymbol{i})) &\leq r \expected d(\widehat{Z}^a_{n}( \boldsymbol{i}), \widehat{Z}^a_{n-1}( \boldsymbol{i}))  
 \end{aligned} 
\end{equation} 
for all $ n $, it follows that $ \widehat{Z}^x_{n}( \boldsymbol{i})$ and $  \widehat{Z}^a_{n}( \boldsymbol{i})  $ a.s.\ converge exponentially fast  to the same limit $ \Pi( \boldsymbol{i})$ (say) by the first inequality in \eqref{eqth0}.  It also follows that \eqref{ineq1}  is simultaneously true with $ \boldsymbol{i} $ replaced by $ i_{k+1}  i_{k+2} i_{k+3}, \dots $ for every $  k $.  It then follows from \eqref{28new}  that 
\[
 \Pi( \boldsymbol{i}) = f_{i_1}\circ \dots \circ f_{i_k} (\Pi( i_{k+1}  i_{k+2}  i_{k+3} \dots)).
\]

Again using \eqref{28new},
\[
\expected d(\widehat{Z}^x_{n}( \boldsymbol{i}),  \Pi( \boldsymbol{i}))
\leq r^n \expected d(x, \Pi( i_{n+1}  i_{n+2}  i_{n+3} \dots))
= r^n \expected d(x, \Pi(  \boldsymbol{i}))
\leq r^n\big( d(x,a) + \expected d(a, \Pi(  \boldsymbol{i}))\big).
\]
But 
\[
 \expected d(a, \Pi(  \boldsymbol{i})) \leq 
 \expected d(a,f_{i_1}(a)) 
       + \sum_{n\geq 1} 
     \expected d(f_{i_1}\circ \dots \circ f_{i_n}(a),f_{i_1}\circ \dots \circ f_{i_{n+1}}(a))
     \leq  \sum_{n\geq 0} r^n L
  = \frac{L}{1-r}.
\]
This gives the  second inequality in \eqref{eqth0}. See \cite{Stenflo98} for details.

The estimates in \eqref{eqth1} are the standard consequence that exponential convergence in mean implies a.s.\ exponential convergence.

\medskip 
\noindent \textbf{b.}  
Suppose $ \phi \in \mathscr{BC}(X,d) $, the   set of bounded continuous functions on $ X $.  Let $ \nu  $ be any unit mass measure. Since for a.e.\ $ \boldsymbol{i}$,
$\widehat{ Z}^x_{n}(\boldsymbol{i}) \to \Pi(\boldsymbol{i}) $ for every $ x$, using the continuity of $ \phi  $ and dominated convergence,
\begin{align*}
\int \phi\, d(F^n \nu) &
=\int \phi \, d(\dist \widehat{Z}^\nu_n)    \text{ (by \eqref{nuphi}) } 
 = \int \phi (\widehat{Z}^x_{n}( \boldsymbol{i}))\,  dW(\boldsymbol{i}) \, d \nu(x)\\
&\to \int \phi(\Pi(\boldsymbol{i}) \,  dW(\boldsymbol{i}) \, d \nu
           = \int \phi\, d\mu,
           \end{align*}
by the definition of  $ \mu $ for the last equality. Thus $F^n(\nu) \to \mu $  weakly.  
The invariance of $ \mu $ and the fact $ \mu $  is the unique invariant unit measure follow from
the weak Feller property.

One can verify that $F:  \mathcal{M}_1(X)  \to   \mathcal{M}_1(X) $ and $  F $ is a contraction map with Lipschitz constant  $ r $ in the $ d_{MK} $ metric. It is easiest to use the second or third form of \eqref{mk} for this. The rest of (\textbf{b}) now follows.  

 \medskip
\noindent \textbf{c.}  
The main difficulty here is that $ (X,d) $ may not be locally compact and so the space $ \mathscr{B} \mathscr{C} (X,  d )$ need not be separable, see Remark~\ref{rmloc}. 
We adapt an idea of Varadhan, see \cite{Dud}*{p399, Thm 11.4.1}.  

There is a totally bounded and hence separable, but not usually complete, metric $ e $ on $ X $ such that $ (X,e) $ and $ (X,d ) $ have the same topology, see  \cite{Dud}*{p72, Thm 2.8.2}.  Moreover, as the proof there shows, $ e(x,y) \leq   d(x,y)$.
Because the topology is preserved, weak convergence of measures on $ (X,d) $  is the same as weak convergence  on $ (X,e) $. 

Let  $\mathscr{B} \mathscr{L}(X,e) $ denote the  set  of bounded Lipschitz functions over $ (X,e) $.  Then $   \mathscr{B} \mathscr{L}(X,e) $ is separable in the \emph{sup  norm} from the total boundedness of $ e $. 

  Suppose $ \phi \in \mathscr{B} \mathscr{L} (X,e) $.  
By the ergodic theorem, since $ \mu  $ is the unique invariant measure for~$ F$, 
\begin{equation} \label{1} 
\int \phi \, d\mu_n^y(\boldsymbol{i}) =
\frac{1}{n} \sum_{k=1}^n \phi(Z_k^y(\boldsymbol{i})) \to \int \phi\,  d\mu 
\end{equation} 
for a.e.\ $ \boldsymbol{i} $   and  $ \mu $ a.e.\ $ y  $.   

 Suppose  $ x \in X $  and choose $ y\in X $  such that \eqref{1} is true. Using   (\textbf{a}),  for a.e.\ $ \boldsymbol{i} $ 
\[ 
e(Z_n^x(\boldsymbol{i}), Z_n^y(\boldsymbol{i})) \leq    d(Z_n^x(\boldsymbol{i}), Z_n^y(\boldsymbol{i}))
  \to 0.  
\]     
It follows from \eqref{1} and the uniform continuity of $ \phi $ in the $  e $  metric  that 
for a.e.\ $ \boldsymbol{i} $, 
\begin{equation} \label{11} 
\int \phi \, d\mu_n^x(\boldsymbol{i})  =
\frac{1}{n} \sum_{k=1}^n \phi(Z_k^x(\boldsymbol{i})) \to \int \phi\,  d\mu.
\end{equation} 

Let $ \mathcal{S} $ be a countable dense subset of $  \mathscr{B} \mathscr{L} (X,e) $ in the sup norm.
One can ensure that \eqref{11} is simultaneously true for all $ \phi \in \mathcal{S}  $.  By an approximation argument it follows  \eqref{11} is simultaneously true for all $  \phi \in  \mathscr{B} \mathscr{L} (X,e) $. 

Since ($ X,e) $ is separable, weak convergence of measures $ \nu_n \to \nu $ is equivalent to 
$ \int \phi \, d\nu_n \to  \int \phi \, d\nu $ for all  $\phi \in \mathscr{B} \mathscr{L} (X,e)  $, see \cite{Dud}*{Thm 11.3.3, p395}. Completeness  is not needed for this.  
It follows that   $ \mu^x_n(\boldsymbol{i}) \to \mu $ weakly as required.

The result \eqref{eqth3} follows from the third inequality in \eqref{eqth1}.

 \medskip \noindent \textbf{d.}
The three inequalities in \eqref{eqth1a} are straightforward as are the claims concerning $ \Pi $.

It follows readily from the definitions that each of $  \mathcal{M}_c(X) $, $  \mathcal{M}_b(X) $,
$ \mathcal{C}(X)  $ and $ \mathcal{B}  \mathcal{C} (X) $ are closed under $  F $,  and  that $  F $ is a contraction map with respect to $ d_P $  in the first two cases and $ d_{  \mathcal{H} }  $  in the second two cases.  The remaining results all follow easily. 
\end{proof}

\begin{remark}[\emph{Contractivity conditions}]\label{rmlog}
The \emph{pointwise} average contractive   condition  
 is implied by the \emph{global} average  
 contractive  condition
$ 
 \expected_\theta r_\theta := 
 \int r_\theta\, dW(\theta) \leq r
$,
where $  
r_\theta := \Lip f_\theta$. Although the global condition is frequently assumed, for our purposes  the  weaker pointwise  assumption is necessary, see Remark~\ref{npc}.

In some papers, for example \cites{DF,WW}, parts of Theorem~\ref{th1} are established or used under the 
 \emph{global log  average
 contractive} and average boundedness conditions 
\begin{equation} \label{} 
 \expected_\theta \log r_\theta <0 ,
  \quad  \expected_\theta  {r_\theta}^q  <\infty,
 \quad  \expected_\theta   d^q (a,f_\theta(a))  <\infty,
\end{equation} 
for some $ q > 0 $ and some $ a \in X $.
However, since  $ d^q $ is a metric for $0< q <1$   and since
$ \left(\expected_\theta g^q(\theta)\right)^{1/q} \downarrow 
   \exp\left(\expected_\theta \log g(\theta) \right) \text { as } q \downarrow  0 $
for $ g\geq 0 $,
such results  follow from Theorem~\ref{th1}.  In the main Theorem~5.2 of \cite{DF} the last two conditions are replaced by the equivalent algebraic tail condition.

One can even obtain in this way similar consequences under the yet weaker \emph{pointwise log average conditions}.    See also \cite{Elton87}. 

Pointwise average contractivity 
 is a much weaker requirement than global average contractivity.  A simple example in which $ f_\theta $ is discontinuous with positive probability is given by $ 0<\epsilon < 1 $, $ X = [0,1]$ with the standard metric $d  $,   $ \Theta = [0,1] $, $ W\{0\}= W\{1\} = \epsilon/2 $, and otherwise $ W $ is  uniformly distributed over $ (0,1) $ according to $ W\{(a,b) \} = (1-\epsilon) (b-a) $ for $ 0<a<b<1 $. Let $ f_{\theta}=\mathcal{X}_{[\theta,1]}$ be the characteristic function of $[\theta,1]$ for $ 0\leq \theta < 1 $ and let $ f_ 1   \equiv 0 $ be the zero function.  Then 
$ \expected_\theta   d(f_\theta(x), f_\theta(y)) \leq (1-\epsilon) d(x,y) $ for all $ x $ and $ y $.   The unique invariant measure is of course $ \frac{1}{2} \delta_0 + \frac{1}{2} \delta_1 $.  Uniqueness fails if $ \epsilon = 0 $.  A simple example where $ f_\theta $ is discontinuous with probability one is 
$ f_\theta  =\mathcal{X}_{ \{1\}}$     and $  \theta $ is chosen uniformly on the unit interval.  See \cites{Stenflo98, Stenflo03-1} for further examples and discussion.

 \end{remark}  

\begin{remark}[\emph{Alternative starting configurations}]
One can extend \eqref{eqth0}   by allowing the starting point  $ x $  to be distributed according to a distribution  $ \nu  $ and  considering the corresponding random variables $ \widehat{Z}^\nu_n(\boldsymbol{i}) $.  Then \eqref{mkcon} follows directly by using the random variables $ \widehat{Z}^\nu_n(\boldsymbol{i}) $ and $ \widehat{Z}^\mu_n(\boldsymbol{i}) $ in  the third form of \eqref{mk}.

Analogous remarks apply in deducing the distributional convergence results in Theorem~\ref{th1}.d from the pointwise convergence results,  via the  third form of \eqref{p}.
 \end{remark}

\begin{remark}[\emph{Local compactness issues}]\label{rmloc}
Versions of Theorem~\ref{th1}.c for locally compact $ (X,d) $ were established in \cites{Breiman, Elton87, BEH, Elton90, MT}.  In that case one first   proves    vague convergence in~\eqref{eqth2}. By  $ \nu_n\to \nu $ vaguely one means $ \int \phi d\nu_n \to \int \phi d\nu $ for all $ \phi \in \mathscr{C}_c(X) $  where $ \mathscr{C}_c(X) $ is the set  of compactly supported continuous functions $ \phi :X\to \mathbb{R} $.
The proof of vague convergence is straightforward from the ergodic theorem since $ \mathscr{C}_c(X) $ is separable.
Moreover, in a locally compact space, 
 vague convergence of probability measures to a probability measure implies weak convergence.  
That this  is not true for more general spaces is a consequence of the following discussion. 

The extension of Theorem~\ref{th1}.c to non locally compact spaces is needed in Section~\ref{secavcon} and Theorem~\ref{th42}.  In order to study images in $\mathbb{R}^k $  we  consider IFS's whose component functions act on  the space $ (\mathcal{M}_1(\mathbb{R}^k)^V,  d_{MK}) $ of  $   V $-tuples of unit mass measures over $ \mathbb{R}^k $, where   $ V $ is a natural number.
Difficulties already arise in proving that the chaos game converges a.s.\ from every initial $ V $-tuple of sets  even for $ V= k =1 $. 

To see this suppose  $ \nu_0 \in  \mathcal{M}_1(\mathbb{R}   )   $ and $ \epsilon > 0  $.  Then   $ B_{\epsilon }(\nu_0) := \{ \nu : d_{MK}(\nu,\nu_0)\leq \epsilon \} $ is not sequentially compact in the $  d_{MK} $ metric and so $  (\mathcal{M}_1(\mathbb{R} ), d_{MK})  $  is not locally compact. To show sequential compactness does not hold
  let 
$ 
\nu_n = \left(1-\dfrac{\epsilon}{n}\right) \nu_0 + \dfrac{\epsilon}{n} \tau_n\nu_0
$,
where $ \tau_n(x) =  x+n  $ is translation by $ n $ units in the $ x $-direction.
Then clearly $ \nu_n \to \nu_0 $ weakly.
Setting $ f(x) = x $ in~\eqref{mk},
\[
d_{MK}(\nu_n,\nu_0) \geq \int x\, d\nu_n - \int x\, d\nu_0\\ 
        =\left(1-\frac{\epsilon}{n}\right)\int x\, d\nu_0 + \frac{\epsilon}{n} \int (x+n)\, d\nu_0
             - \int x\, d\nu_0 \\
             = \epsilon. 
\]
On the other hand, let $ W $ be a random measure with  $ \dist W = \nu_0 $.  Independently of the value of $ W $ let $ W'=W $ with probability $ 1-\dfrac{\epsilon}{n} $ and $ W' = \tau_n W $ with probability $ \dfrac{\epsilon}{n} $.  Then again from~\eqref{mk}, 
\[
d_{MK}(\nu_n,\nu_0) \leq \expected  d_{MK}(W,W') = 
     \left(1-\frac{\epsilon}{n}\right)\times 0 + \frac{\epsilon}{n}\, n = \epsilon.
     \]
 It follows that $ \nu_n \in B_{\epsilon }(\nu_0) $ and $ \nu_n \nrightarrow \nu $  in the $ d_{MK} $ metric, nor does any subsequence.  Since $ d_{MK} $ implies weak convergence it follows that $ (\nu_n)_{n\geq 1} $ has no convergent subsequence in the $ d_{MK} $ metric.
 
 It follows that  $ \mathscr{C}_c( \mathcal{M}_1(\mathbb{R} ),  d_{MK}   ) $ contains only the zero function.  Hence  vague convergence in this setting is a vacuous notion and  gives no information about weak convergence.

Finally, we note that although (c) is proved here assuming the pointwise average contractive condition, it is clear that weaker  hypotheses concerning the stability of trajectories will suffice to extend known results from the locally compact setting.
\end{remark}
 
\begin{remark}[\emph{Separability and measurability issues}]\label{nosep}  If $ (X,d) $ is separable then   the class of Borel sets for the product topology on $ X \times X $ is the product   $ \sigma $-algebra of the class of Borel sets on $ X $ with itself, see~\cite{Bil}*{p 244}.  It follows that $ \theta \mapsto  d(f_\theta(x), f_\theta(y)) $ is measurable for each $ x,y \in X $ and so the quantities in~\eqref{eqmc}  are well defined.

Separability is not required
for the uniform contractive and uniform boundedness conditions  in \eqref{equcb} and  the conclusions in (\textbf{d}) are still valid with essentially the same proofs.  The spaces $ \mathcal{M}_c(X) $,  $ \mathcal{M}_b(X) $ and $ \mathcal{M}_1(X) $ need to be restricted to separable measures as discussed following Definition~\eqref{mx2} and Proposition~\ref{prop25}.

Separability   is also used  in the proof of (\textbf{c}).  
If one drops this condition   and assumes the uniform contractive and boundedness conditions \eqref{equcb} then a weaker version  of  (\textbf{c}) holds.  Namely, 
for every $ x\in X $ and every bounded continuous $ \phi\in \mathscr{B} \mathscr{C} (X)$, for   a.e.\  $ \boldsymbol{i} $ 
\begin{equation} \label{tfc} 
\int \phi \, d\mu^x_n(\boldsymbol{i}) \to \int \phi\,  d\mu.
\end{equation} 
The point is that unlike the situation in (\textbf{c}) under the hypothesis of separability, the set of such $ \boldsymbol{i} $ might depend on the function $ \phi $.

In Theorem~\ref{th41} we   apply Theorem~\ref{th1} to an IFS  whose component functions operate on  $ (\mathcal{M}_c(X)^V,  d_{P}) $ where   $ V $ is a natural number.  Even in the case $ V=1 $ this   space is not separable, see Example~\ref{nsp}. 

 \end{remark}

\section{Tree Codes and Standard Random  Fractals} \label{secifsrf} 

Let  $ \boldsymbol{F} = \{X, F^\lambda, \lambda  \in\Lambda, P\}   $ be a family of IFSs
as in \eqref{dfF1} and~\eqref{dfF2}.
We   assume the IFSs $  F^\lambda  $ are  \emph{uniformly contractive} and \emph{uniformly bounded}, i.e.\ for some   $0< r<1$, 
\begin{equation} \label{equc} 
\sup\nolimits_\lambda \max\nolimits_{m} d(f_m^\lambda(x), f_m^\lambda(y)) \leq r \,d(x,y) 
\quad \text{and} \quad L:=\sup\nolimits_\lambda \max\nolimits_m  d(f^\lambda_m(a),a) <\infty
\end{equation} 
for all $ x,y\in X  $  and some $ a\in X $.  
More general conditions are assumed in Section~\ref{secavcon}. 

\medskip

We often use $ * $ to indicate concatenation of sequences, either finite or infinite.

\begin{definition}[\emph{Tree codes}]\label{dftrees} 
The tree $ T $ is the  set of all  finite sequences from $ \{1,\dots,M\} $, including the empty sequence $ \emptyset $.
If $ \sigma = \sigma_1\dots\sigma_k \in T $ then the \emph{length} of $\sigma $ is $|\sigma |=k $  and $|\emptyset|=0 $.

A \emph{tree code} $ \omega $ is a map $ \omega : T \to \Lambda $.  The metric space $ (\Omega ,d ) $  of all tree codes  is defined by 
\begin{equation} \label{eqtrees}
\Omega = \{\omega \mid  \omega : T \to \Lambda\}, \quad  d( \omega, \omega')   = \frac{1}{M^k} 
\end{equation}
if $ \omega(\sigma) = \omega'(\sigma) $ for all $ \sigma $ with $ |\sigma|< k $ and 
$ \omega(\sigma) \neq \omega'(\sigma) $ for some $ \sigma $ with $ |\sigma|= k $.

A \emph{finite tree code  of height $  k $} is a map $ \omega: \{\sigma \in T : |\sigma | \leq k \} \to \Lambda$.   

If $ \omega \in \Omega $ and $ \tau \in T $ then the tree code $  \omega \rfloor\tau$ is defined by $  (\omega \rfloor \tau)\, (\sigma) :=    \omega ( \tau * \sigma )$.  It is the tree code obtained from $ \omega $ by starting at the node $ \tau $.  One similarly defines    $  \omega \rfloor\tau$ if $ \omega $ is a finite tree code of height $ k  $ and $ |\tau | \leq k $.

If $ \omega \in \Omega $ and $ k $ is a natural number then the finite tree code $ \omega\lfloor k $ defined by  $ (\omega\lfloor k)(\sigma) = \omega (\sigma) $  for $|\sigma |\leq k $.  It is obtained by truncating $ \omega $ at the level $ k $.
\end{definition}

The space   $(\Omega, d)$ is complete and bounded.  
If $  \Lambda $ is finite  then $(\Omega, d)$  is compact.   

The tree code $ \omega $ associates to each node   $ \sigma \in T $ the IFS $ F^{\omega(\sigma)}$.  It also associates to each   $ \sigma \neq \emptyset $ the function $ f_m^{\omega(\sigma')} $ where $ \sigma  =\sigma' *m $.  The $ M^k $ components of $ K_k^\omega   $ in \eqref{dKo} are then obtained by beginning with the set $ K_0  $ and iteratively applying the functions associated to the $ k $ nodes along each of the $ M^k $ branches of depth $  k $ obtained from $ \omega $.  A similar remark applies to $ \mu_k^\omega   $.

\begin{definition}[\emph{Fractal sets and measures}]\label{def28}
If  $ K_0 \in \mathcal{C} ( X) $ and $ \mu_0 \in \mathcal{M}_c(X) $ then the \emph{prefractal sets}  $ K_k^\omega $, the \emph{prefractal measures}  $ \mu_k^{\omega} $, the
\emph{fractal set} $ K^\omega $ and the \emph{fractal measure} $ \mu^{\omega} $,  are given by
\begin{equation} \label{dKo}
\begin{aligned}
K_k^\omega 
&=  \bigcup_{\sigma  \in T,\, |\sigma |=k} 
  f^{\omega(\emptyset)}_{\sigma_1}\circ f^{\omega(\sigma_1)}_{\sigma_2} 
    \circ f^{\omega(\sigma_1\sigma_2)}_{\sigma_3}\circ \dots    \circ  f^{\omega(\sigma_1\dots \sigma_{k-1})}_{\sigma_k}  (K_0), \\
 \mu_k^\omega 
&=    \sum_{\sigma  \in T,\, |\sigma |=k} 
 w^{\omega(\emptyset)}_{\sigma_1} w^{\omega(\sigma_1)}_{\sigma_2} 
   \cdot \ldots    \cdot  w^{\omega(\sigma_1\dots \sigma_{k-1})}_{\sigma_k}
     \, 
  f^{\omega(\emptyset)}_{\sigma_1}\circ f^{\omega(\sigma_1)}_{\sigma_2} 
   \circ \dots    \circ  f^{\omega(\sigma_1\dots \sigma_{k-1})}_{\sigma_k}  (\mu_0) ,\\ 
   K^\omega 
 &= \lim_{k\to \infty} K_k^\omega, \qquad 
\mu^\omega = \lim_{k\to \infty} \mu_k^\omega.
    \end{aligned} 
\end{equation} 
It follows from  uniform contractivity that for \emph{all}  $ \omega $ one has convergence in the Hausdorff and strong Prokhorov metrics respectively, and that
 $   K^\omega $ and $  \mu^\omega  $  are  independent of $ K_0 $ and $ \mu_0 $.  
 
 The collections of all such  fractals sets and measures for fixed  $ \{F^\lambda\}_{\lambda \in \Lambda } $ 
 are denoted  by 
\begin{equation} \label{eq29}
 \mathcal{K}_\infty = \{ K^\omega : \omega \in \Omega \} 
 , \quad     \mathcal{M}_\infty = \{ \mu^\omega : \omega \in \Omega \}. 
\end{equation}  

For each~$ k $ one has
\begin{equation} \label{subf}
K^\omega   = \bigcup_{| \sigma |=k} K^\omega_{\sigma },   \text{ where }
K^\omega_\sigma :=
  f^{\omega(\emptyset)}_{\sigma_1}\circ f^{\omega(\sigma_1)}_{\sigma_2} 
    \circ f^{\omega(\sigma_1\sigma_2)}_{\sigma_3}\circ \dots    \circ  f^{\omega(\sigma_1\dots \sigma_{k-1})}_{\sigma_k}  
    (K^{\omega \rfloor \sigma }).
\end{equation} 
 The $ M^k $ sets 
$ K^\omega_{\sigma } $ are called the \emph{subfractals} of $ K^\omega $  at level $ k $.
\end{definition} 

The maps  $ \omega \mapsto K^\omega   $  and $ \omega \mapsto \mu^\omega $ are   H\"{o}lder continuous.  More precisely: 
\begin{proposition}\label{prpdkw}
With $  L$  and $ r $  as in \eqref{equc},
\begin{equation} \label{dkw}
 d_{ \mathcal{H} }(K^\omega,K^{\omega'}) \leq \frac{2L}{1-r}\,  d^\alpha (\omega,\omega') \quad 
 \text{and} \quad d_{ P }(\mu^\omega,\mu^{\omega'}) \leq \frac{2L}{1-r}\,  d^\alpha (\omega,\omega'),
 \end{equation} 
 where $ \alpha  = \log (1/r) / \log M $.  
\end{proposition}

\begin{proof}
Applying  \eqref{dKo} with  $ K_0 $ replaced by $  \{a\}$, and using \eqref{equc} and repeated applications of the triangle inequality, it follows that $ d_{\mathcal{H}}(K^\omega, a) \leq (1+r + r^2 +\cdots) L = L/(1-r) $ and so $ d_{\mathcal{H}} (K^\omega, K^{\omega'})
\leq 2L/(1-r) $ for any $ \omega $  and $ \omega'$.  
If $ d(\omega,\omega') = M^{-k} $  then
$ \omega(\sigma )  = \omega'( \sigma ) $ for $ |\sigma |< k $, and since $  d_{\mathcal{H}} (K^{\omega\rfloor \sigma}, K^{\omega'\rfloor \sigma}) \leq 2L/(1-r)$, it follows from \eqref{subf} and contractivity that $ d_{ \mathcal{H} }(K^\omega,K^{\omega'}) \leq \frac{2Lr^k}{1-r} $.  Since $ r^k = M^{-k\alpha} = d^\alpha (\omega,\omega') $, the  
 result for sets  follows.  
 
 The proof for measures is essentially identical; one replaces $ \mu_0 $ by $ \delta_{a} $ in  \eqref{dKo}.
 \end{proof}
 
\begin{figure}[htbp] 
   \centering
   \includegraphics[width=3.6in]{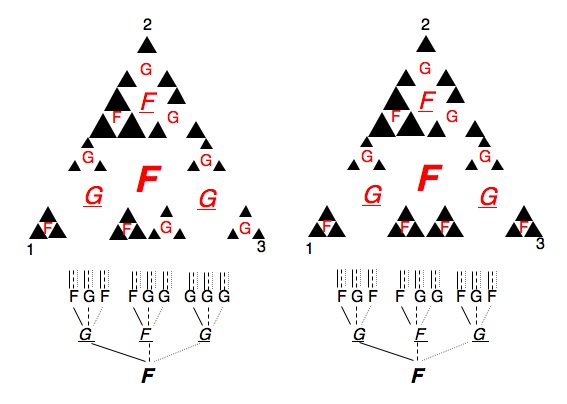} 
   \caption{Random Sierpinski triangles and tree codes}
   \label{fig:example}
\end{figure}

\begin{example}[\emph{Random Sierpinski triangles and tree codes}]
The relation between a tree  code $\omega $ and the corresponding fractal set $ K^\omega $ can readily be seen in Figure~\ref{fig:example}.  The IFSs  $F=(f_1,f_2,f_3) $ and $ G=(g_1,g_2,g_3) $ act on $ \mathbb{R}^2 $, and the $ f_m $  and  $g_m$  are similitudes with contraction ratios 1/2 and  1/3 respectively  and fixed point $ m $.  If a node is labelled $F$, then reading from left to right  the three main branches of the subtree associated with that node correspond   to $f_1, f_2, f_3 $ respectively.   Similar remarks apply  if the  node is labelled~$ G $.

If the three functions in $F$ and the three functions in $G$ each are given weights equal to $1/3 $ then the measure $ \mu^\omega $ is distributed over $ K^\omega $ in such a way that $ 1/3 $ of the mass is in each of the top level triangles, $1/9 $ in each of the next level triangles, etc.   
\end{example}

The sets $ K^\omega $ and measures $ \mu^\omega $ in Definition~\ref{def28} are not normally self similar in any natural sense.  However, there is an associated notion of statistical self similarity.  For this we need the following definition.

The  reason for the notation  $\rho_{\infty} $ in the following definition can be seen from Theorem~\ref{rc}. 

\begin{definition}[\emph{Standard random fractals}]\label{def211} The \emph{probability distribution $ \rho_\infty $ on $ \Omega $}  is defined  by choosing $ \omega (\sigma) \in \Lambda $ for each $ \sigma \in T $ in an iid manner according to  $ P $.
 
 The \emph{random set} $ \mathbb{K} = \omega\mapsto K^\omega   $ and the \emph{random measure} $ \mathbb{M} = \omega\mapsto \mu^\omega   $, each defined by choosing $ \omega \in \Omega $ according to ~$ \rho_\infty $, are called \emph{standard random fractals}.

 The induced \emph{probability distributions  on $  \mathcal{K}_\infty $ and $  \mathcal{M}_\infty $}  respectively are defined by  $ \mathfrak{K}_\infty = \dist  \mathbb{K}  $
 and $ \mathfrak{M}_\infty = \dist  \mathbb{M}  $.
\end{definition}

It follows from the definitions that $ \mathbb{K}   $ and $ \mathbb{M}   $ are  \emph{statistically  self similar} in the sense that 
\begin{equation} \label{IFSro} 
\dist \mathbb{K}  = \dist \mathbb{F}  (\mathbb{K}_1,\dots, \mathbb{K}_M), \quad
\dist \mathbb{M}  = \dist \mathbb{F}  (\mathbb{\mathbb{M} }_1,\dots, \mathbb{\mathbb{M} }_M),
\end{equation} 
where $ \mathbb{F}   $ is a random IFS chosen from $(F^\lambda)_{\lambda\in \Lambda} $ according to $P$,
$\mathbb{K}_1,\dots, \mathbb{K}_M $ are iid copies of $\mathbb{K} $ which   are independent of $ F $,  
and $\mathbb{M}_1,\dots,\mathbb{M}_M $ are iid copies of $\mathbb{M} $ which   are independent of $ \mathbb{F}   $.  

Here, and in the following sections, an IFS    $ F$ acts on $ M $-tuples of subsets    
 $ K_1,\dots, K_M $ of $  X $ and measures $  \mu_1,\dots, \mu_M $ over $ X $ by 
\begin{equation} \label{eq21a}
 F(K_1,\dots,K_M)  =  \bigcup_{m=1}^M f_m(K_m), \quad 
 F(\mu_1,\dots,\mu_M)  = \sum_{m=1}^M w_m f_m(\mu_m).
 \end{equation}
 This extends in a pointwise manner to random IFSs acting on random sets and random measures as in~\eqref{IFSro}. 

We use the terminology ``standard'' to distinguish the class of random fractals given by Definition~\ref{def211} and discussed in (\cites{Falconer86, MW86, Graf87,  HR98, HR00}) from other classes of random fractals in the literature.  

\section{$ V $-Variable Tree Codes}\label{secvvtc}

\subsection{Overview} We continue with the assumptions that  $ \boldsymbol{F} = \{X, F^\lambda, \lambda  \in\Lambda, P\}   $ is a family of IFSs
as in \eqref{dfF1} and~\eqref{dfF2}, and that $ \{F^ \lambda \}_{\lambda \in \Lambda } $  satisfies  the  uniform contractive  and  uniform bounded conditions~\eqref{equc}. 
In Theorem~\ref{th42} and Example~\ref{rm54} the uniformity conditions are replaced by pointwise average conditions. 

In Section~\ref{secvvar2}, Definition~\ref{Vvar}, we define  the set $ \Omega_V \subset \Omega $ of $ V $-variable tree codes, where $ \Omega $ is the set of tree codes in Definition~\ref{dftrees}.  Since  the $ \{ F^\lambda \}_{\lambda \in \Lambda} $ are uniformly contractive this leads directly to the class   $ \mathcal{K}_V $ of $ V $-variable fractal sets and the class $ \mathcal{M}_V $ of $ V $-variable fractal measures.

In Section~\ref{secifsv},  $ \Omega_V $ is alternatively  obtained from an IFS $ \boldsymbol{\Phi}_V = (\Omega^V, \Phi^a, a \in \mathcal{A}_V )$ acting on $ \Omega^V $.  More precisely, the attractor $ \Omega_V^* \subset \Omega^V $ of $\boldsymbol{\Phi}_V $ projects  in any of the $V$-coordinate directions
to $ \Omega_V $.  However, $   \Omega_V^* \neq (\Omega_V)^V$ and in fact there is a high degree of dependence between the coordinates of any $\boldsymbol{\omega }  =
 (\omega_1,\dots,\omega_V) \in \Omega_V^*$.
 
If 
\[
\boldsymbol{\omega}   =
 \lim_{k \to \infty}  \Phi^{a_0} \circ  \Phi^{a_1} \circ  \dots \circ  \Phi^{a_{k-1}}(\omega^0_1,\dots,\omega^0_V) 
\]
we say $   \boldsymbol{\omega}$ has address $ a_0a_1\dots a_k \dots  $\,.  The limit is independent of $   (\omega^0_1,\dots,\omega^0_V)$.

In Section~\ref{sec34}  a formalism is developed for finding the $ V $-tuple of tree  codes  $  (\omega_1,\dots,\omega_V) $ from the address $ a_0a_1\dots a_k \dots  $, see  Proposition~\ref{adct} and Example~\ref{rem17}.   Conversely, given a tree code  $  \omega $ one can find all possible addresses  $ a_0a_1\dots a_k \dots  $ of $ V $-tuples $  (\omega_1,\dots,\omega_V)\in  \Omega_V^*$ for which $ \omega_1=\omega $.

In Section~\ref{secpdtc}  the probability distribution $ \rho_V $ on the set $ \Omega_V $ of $ V $-variable tree codes is defined and discussed.   The probability $ P $ on $ \Lambda $ first leads to a natural probability $ P_V $ on the index set $ \mathcal{A}_V $ for the IFS $ \boldsymbol{\Phi}_V $, see Definition~\ref{df38}.   This then turns $  \boldsymbol{\Phi}_V $ into an IFS $  (\Omega^V, \Phi^a, a \in \mathcal{A}_V, P_V )$ with weights whose measure attractor  $ \rho_V^* $ is a probability distribution on its set attractor $\Omega_V^* $.  The projection of $ \rho_V^* $  in any coordinate direction is the same, is supported on $ \Omega_V $ and is denoted by $ \rho_V $, see Theorem~\ref{macs}.

\subsection{$V $-Variability}\label{secvvar2}

\begin{definition}[\emph{$ V $-variable  tree codes and fractals}] \label{Vvar}  A  tree code $ \omega\in\Omega$ is   \emph{$ V$-variable}   if for each  positive integer  $k $  there are   at most $ V $ distinct tree codes of the form $\omega \rfloor \tau $ with $|\tau |= k $.    The set of $V$-variable tree codes is denoted by~$\Omega_V $. 

Similarly a finite tree code  $ \omega $  of height $  p $  is \emph{$ V $-variable} if for  each $ k< p $ there are  at most $ V $ distinct finite subtree codes $\omega \rfloor \tau $ with $|\tau |= k $.   

For a uniformly contractive family $ \{F^\lambda\}_{\lambda \in \Lambda} $ of IFSs,
if $ \omega $ is $ V $-variable then  the fractal set $ K^\omega $ and fractal measure $ \mu^\omega $ in \eqref{dKo} are said to be    \emph{$V$-variable}.  The collections of all $ V $-variable sets and measures corresponding to $ \{F^\lambda\}_{\lambda \in \Lambda} $ are denoted by
\begin{equation} \label{eq31a} 
 \mathcal{K}_V= \{K^\omega : \omega \in \Omega_V \}, \quad 
  \mathcal{M}_V= \{\mu^\omega : \omega \in \Omega_V \}   
\end{equation}  
respectively, c.f.\ \eqref{eq29}.
\end{definition}

 If $ V=1 $  then $ \omega $ is $  V$-variable if and only if $ |\sigma| = |\sigma'| $ implies $\omega(\sigma) = \omega (\sigma') $, i.e.\ if and only if for each  $ k  $ all values of  $ \omega(\sigma) $  at level $k  $  are equal. In the case   $ V>1  $, if $ \omega $ is $  V$-variable then for each  $ k  $ there are at most $  V $  distinct values of $ \omega(\sigma) $  at level $k= |\sigma| $, but this is  not sufficient to imply $ V $-variability.

\begin{remark}[\emph{$ V $-variable terminology}]  \label{rm32}  
The motivation for the terminology ``$ V $-variable fractal''   is as follows. 
Suppose all functions $ f^\lambda_m $ belong to the same group $ G $  of transformations.  For example, if $X = \mathbb{R}^n  $ then $ G$  might be the group of invertible similitudes, invertible affine transformations or invertible projective transformations.  Two sets $ A $  and $  B $  are said to be  \emph{equivalent modulo  $ G $}  if $ A = g(B) $ for some $  g\in G $.
If  $ K^\omega $  is $ V$-variable and $ k $ is a positive integer, then there are at most $ V $ distinct trees of the form   $\omega \rfloor \sigma$ such that $ | \sigma | = k $.  If  $ | \sigma | = | \sigma' | = k $  and  
$\omega \rfloor \sigma =   \omega \rfloor \sigma' $, then from \eqref{subf}
\begin{equation} \label{vvarsets}
K^\omega_\sigma =  g(K^\omega_{\sigma'}) \text{ where }
g = 
 f^{\omega(\emptyset)}_{\sigma_1}\circ  \dots    
 \circ  f^{\omega(\sigma_1\dots \sigma_{k-1})}_{\sigma_k}  
 \circ  \left(f^{\omega({\sigma}'_1\dots{ \sigma}'_{k-1})}_{{\sigma}'_k}\right)^{-1}
 \circ \dots
 \circ    \left(f^{\omega(\emptyset)}_{{\sigma}'_1}\right)^{-1}
K^\omega_{\sigma'}.
\end{equation} 
In particular, $ K^\omega_\sigma $ and $   K^\omega_{\sigma'}$ are equivalent modulo $  G $.  

Thus  the subfractals of $ K^\omega $ at level $  k $ form at most $ V $ distinct equivalence classes modulo  $ G $.  However, the actual equivalence classes depend upon the  level.

Similar remarks apply to $ V $-variable fractal measures.
  \end{remark}

\begin{proposition} A tree code $ \omega $  is $ V $-variable iff for every positive integer $ k $  the finite tree codes $ \omega \lfloor k $ are  $ V $-variable.
\end{proposition}

\begin{proof}  If $ \omega $  is $  V $-variable   the same is true for every finite tree code of the form  $ \omega \lfloor k $.

If $ \omega $  is not $  V $-variable then for some $  k $ there are at least $ V+1 $ distinct subtree codes $ \omega \rfloor \tau $  with $ | \tau | = k $. But then for some  $  p$  the $ V + 1 $ corresponding finite tree codes $ (\omega \rfloor \tau)\lfloor p $    must also be distinct.  It follows  $ \omega\lfloor  (k+ p) $ is not $ V $-variable.
\end{proof}

\begin{example}[\emph{$ V $-variable Sierpinski triangles}]
The first tree code in Figure~\ref{fig:example} is an initial segment of a 3-variable tree code but not of a  2-variable tree code, while the second tree is an initial segment of a 2-variable tree code but not of a  1-variable tree code.
The corresponding Sierpinski type triangles are, to the level of approximation shown, 3-variable and 2-variable respectively.
\end{example}

\begin{theorem} \label{prop33}
The $ \Omega_V $ are closed and nowhere dense in $ \Omega $, and
\[
\Omega_V\subset \Omega_{V+1}, 
 \quad d_{\mathcal{H}}(\Omega_V,\Omega ) <\frac{1}{V},
 \quad \bigcup_{V\geq 1}\Omega_V\subsetneqq
 \overline{\bigcup_{V\geq 1}\Omega_V}=\Omega,
 \] 
where the bar denotes closure in the metric $ d $. 
\end{theorem}

\begin{proof}
For the inequality suppose $ \omega \in  \Omega $ and  define  $ k $ by $ M^{k}\leq V < M^{k+1}$.  Then if $  \omega' $ is chosen so $ \omega' ( \sigma  )   = \omega(\sigma) $ for $ | \sigma |\leq k $ and 
$ \omega' ( \sigma  )  $ is constant for $ | \sigma |> k $, it follows 
$  \omega' \in \Omega_V $ and $ d( \omega', \omega) \leq M^{-(k+1)}<V^{-1} $, hence 
$ d_{\mathcal{H}}(\Omega_V,\Omega ) <V^{-1} $.  The remaining assertions are clear.
\end{proof}
 
\subsection{An IFS Acting on $ V $-Tuples of  Tree Codes}\label{secifsv}
\begin{definition}\label{omv}
The metric space $ (\Omega^V, d) $ is the set of $ V $-tuples from $ \Omega $  with the metric
\begin{equation*} \label{}
 d\big((\omega_1\dots \omega_V), (\omega'_1\dots \omega'_V)\big) 
             = \max_{1\leq v \leq V} d( \omega_v, \omega_v') ,
             \end{equation*} 
where $ d$  on the right side is as in Definition~\ref{dftrees}.
\end{definition}
This is a   complete bounded  metric  and  is compact if $  \Lambda  $ is finite since the same its true for $ V=1 $. See Definition~\ref{dftrees} and the comment which follows it. 
   The induced Hausdorff metric on $ \mathcal{B} \mathcal{C} ( \Omega^V) $ is complete and bounded, and is compact if $  \Lambda  $ is finite.  See the comments following Definition~\ref{cx}.  

\medskip
The notion of $V$-variability extends to  $ V $-tuples of tree codes, $ V $-tuples of sets and $ V $-tuples of measures.
\begin{definition}[\emph{$ V $-variable $ V $-tuples}]\label{df34} The $ V $-tuple of tree codes
$ \boldsymbol{\omega}  = ( \omega_1, \dots , \omega_V) \in \Omega^V$     is    \emph{$ V $-variable}   if  for each positive integer $ k $  there are  at most $ V $ distinct  subtrees   of the form $\omega_v \rfloor \sigma $ with $ v \in \{1,\dots,V\} $ and $|\sigma |= k $.  The \emph{set of $ V $-variable $ V $-tuples of tree codes} is denoted by $ \Omega_V^*$.

Let  $ \{F^\lambda\}_{\lambda \in \Lambda} $ be  a uniformly contractive family of IFSs. The corresponding   sets $ \mathcal{K}^*_V $ of $ V $-variable $ V $-tuples of fractal sets,  and   $ \mathcal{M}^*_V $ of $ V $-variable $ V $-tuples of fractal measures,  are
\[  \mathcal{K}^*_V= 
\{(K^{\omega_1},\dots,K^{\omega_V}) : (\omega_1,\dots,\omega_V) \in \Omega_V^* \} , \quad 
      \mathcal{M}^*_V= 
\{(\mu^{\omega_1},\dots,\mu^{\omega_V}) : (\omega_1,\dots,\omega_V) \in \Omega_V^* \} ,
\]
where $ K^{\omega_v} $ and $ \mu^{\omega_v}$ are as in Definition~\ref{def28}.
\end{definition}

\begin{proposition}
The projection of $   \Omega_V^*$ in any coordinate direction equals $ \Omega_V $, however \mbox{$   \Omega_V^* \subsetneqq (\Omega_V)^V$}.  
\end{proposition}

\begin{proof}
To see the projection map is onto consider $ (\omega,\dots,\omega) $ for $ \omega \in \Omega_V $.  To see $   \Omega_V^* \subsetneqq (\Omega_V)^V$ note that a $ V $-tuple of $ V $-variable tree codes need not itself be $ V $-variable.
\end{proof}

\begin{notation}\label{concat}
Given $ \lambda \in \Lambda $ and $ \omega_1,\dots, \omega_M \in \Omega $ define $  \omega = \lambda *( \omega_1,\dots, \omega_M)  \in   \Omega $ by
$ 
  \omega(\emptyset) =  \lambda $  and $  \omega(m\sigma) = \omega_m(\sigma)   
 $. 
 Thus $  \lambda *( \omega_1,\dots, \omega_M)$ is the tree code with $  \lambda $ at the base node $ \emptyset $ and the tree  $ \omega_m $ attached to the node  $  m $  for $ m=1,\dots, M $.
 
 Similar notation applies if the $ \omega_1,\dots,\omega_M $  are finite tree codes all of the same height.
 \end{notation}
 
We define maps on $ V$-tuples of tree codes and a corresponding IFS on $ \Omega^V  $   as follows.
\begin{definition}[\emph{The IFS acting on the set of $  V $-tuples of tree codes}]\label{dfIFStr}
Let $ V $ be  a positive integer. Let   $\mathcal {A}_V  $ be the  set  of all pairs of maps  $ a =(I,J) = (I^a,J^a)$, where   
\[
 I:\{1,\dots,V\}\to \Lambda, \quad  J:\{1,\dots,V\}\times \{1,\dots,M\} \to \{1,\dots,V\} .
 \] 
For $ a \in \mathcal{A}_V $ the map $ \Phi^a    : \Omega^V \to \Omega ^V $ is defined for $  \boldsymbol{\omega} = (\omega_1,\dots,\omega_V) $
by 
\begin{equation} \label{phia}
\Phi^a (\boldsymbol{\omega})  = (\Phi^a_1(\boldsymbol{\omega}),\dots, \Phi^a_V(\boldsymbol{\omega}) ), \quad 
\Phi^a_v (\boldsymbol{\omega})  = I^a (v) * (\omega_{J^a (v,1)},\dots, \omega_{J^a (v,M)}).
\end{equation}  
Thus $\Phi_v^a (\boldsymbol{\omega}) $ is the tree code with base node $ I^a (v)$, and 
at the end of  each of its $ M $  base branches
are attached  copies of $ \omega_{J^a (v,1)},\dots, \omega_{J^a (v,M)}   $ respectively.

The  \emph{IFS $ \boldsymbol{\Phi}_V$ acting on $  V $-tuples of tree codes} and without a probability distribution at this stage is defined by
\begin{equation} \label{trifs}
\boldsymbol{\Phi}_V:=(\Omega^V, \Phi^a, a \in \mathcal{A}_V ) .
\end{equation} 
\end{definition}

\emph{Note that   $\Phi^a: \Omega_V^* \to \Omega_V^* $} for each $ a \in \mathcal{A}_V $.    

\begin{notation}
It is often convenient to write $ a =(I^a,J^a) \in \mathcal {A}_V$  in the form
\begin{equation} \label{nIJ} 
a = \begin{bmatrix} I^a (1) & J^a (1,1) & \dots &J^a (1,M) \\
                                   \vdots & \vdots & \ddots  & \vdots \\
                                   I^a (V) & J^a (V,1) & \dots & J^a (V,M) 
       \end{bmatrix}.
       \end{equation} 
Thus $ \mathcal{A}_V $ is then the set of all $ V \times (1+M) $ matrices  with entries in the first column belonging to $ \Lambda  $ and all other entries belonging to $ \{1,\dots,V\} $.
\end{notation}

\begin{theorem}\label{char2}
Suppose $ \boldsymbol{\Phi}_V = (\Omega^V, \Phi^a, a \in \mathcal{A}_V )$ is an IFS  as in Definition~\ref{dfIFStr}, with $ \Lambda $ possibly infinite. Then each $ \Phi^a $ is a contraction map with Lipschitz constant $ 1 /M $.  Moreover, with $ \boldsymbol{\Phi}_V  $ acting on subsets of $ \Omega^V $ as in~\eqref{eq21} and using the notation of Definition~\ref{cx}, we have 
$ \boldsymbol{\Phi}_V  : \mathcal{B} \mathcal{C} (\Omega^V) \to \mathcal{B} \mathcal{C} (\Omega^V) $ and $ \boldsymbol{\Phi}_V$ is a contractive map with Lipschitz constant 1/M.  The unique fixed point of $ \boldsymbol{\Phi}_V$  is $ \Omega_V^* $ and in particular its projection in any coordinate direction equals~$ \Omega_V $.
\end{theorem}

\begin{proof} 
It is readily checked that  each $ \Phi^a $ is a contraction map with Lipschitz constant $ 1 /M $.

We can establish directly  that $  \boldsymbol{\Phi}_V(E) := \bigcup_{a\in  \mathcal{A}_V} \Phi^a(E)$  
is closed  if $ E $ is closed, since any Cauchy sequence from 
$  \boldsymbol{\Phi}_V(E) $ eventually belongs to $  \Phi^a(E)$ for some fixed $ a $.  It follows that  $  \boldsymbol{\Phi}_V $ is a contraction map on the complete space
$ (\mathcal{B} \mathcal{C} ( \Omega^V),d_{\mathcal{H}}) $ with Lipschitz constant $ 1/M $ and so has a unique bounded closed fixed point (i.e.\ attractor).

In order to show this attractor is the set $ \Omega_V^* $ from Definition~\ref{df34}, note that   $ \Omega_V^* $  is bounded and closed in $ \Omega^V $. It is closed under  $ \Phi^a $ for any $ a \in \mathcal{A}_V $ as noted before.  Moreover, each $\boldsymbol{\omega}  \in \Omega_V^*  $ is of the form $  \Phi^a(\boldsymbol{\omega}') $ for some   $\boldsymbol{\omega}'\in \Omega_V^* $ and some (in fact many)  $ \Phi^a $.  
To see this, consider the $ VM $ tree codes of the form $ \omega_v\rfloor m $ for $ 1 \leq v \leq V $ and $ 1 \leq m \leq M $, where each $ m $ is the corresponding node of $ T $ of height one.  There are at most $ V $ distinct such tree codes, which we denote by $ \omega'_1, \dots, \omega'_V$,  possibly with repetitions.  Then from \eqref{phia} 
\[
(\omega_1,\dots, \omega_V ) = \Phi^a(\omega'_1,\dots,\omega'_V),
\]
provided
\[
I^a(v) = \omega_v(\emptyset), \quad \omega'_{J^a(v,m)} = \omega_v\lfloor m.
\]

So $\Omega_V^*  $ is invariant under  $  \boldsymbol{\Phi}_V$ and hence is the unique attractor of the IFS $  \boldsymbol{\Phi}_V$.
\end{proof}

In the previous theorem, although $ \boldsymbol{\Phi}_V$ is an IFS, neither Theorem~\ref{th1} nor the extensions in   Remark~\ref{nosep} apply directly.  
If $ \Lambda $  is not finite then $ \Omega^V  $ is neither separable nor compact.  Moreover, the map $ \boldsymbol{\Phi}_V$ acts on sets by taking infinite  unions and so   we cannot apply Theorem~\eqref{th1}.d to find a set attractor for $ \boldsymbol{\Phi}_V$, since in general the union of an infinite number of closed sets need not be closed.  

As a consequence of the theorem, approximations to $ V $-variable $ V $-tuples of tree codes, and in particular to individual $ V $-variable  tree codes,   can be built up  from a $ V $-tuple $ \boldsymbol{\tau} $ of  finite    tree codes of height 0 such as $\tau =  (\lambda^*, \dots ,  \lambda^*)$ for some $ \lambda^*\in \Lambda $,
and a finite sequence  $ a_0, a_2,\dots , a_k \in \mathcal{A} _V $, by computing the height $ k$ finite tree 
code $\Phi^{a_0}\circ\dots\circ  \Phi^{a_k} (\boldsymbol{\tau}) $.   Here we use the natural analogue of \eqref{phia} for finite tree codes.   See also the diagrams in \cite{BHS05}*{Figures 19,20}.

\subsection{Constructing Tree Codes from Addresses and Conversely}\label{sec34}

\begin{definition}[\emph{Addresses for $ V $-variable $ V $-tuples of tree codes}]\label{dfseqcode} 
For each sequence  $ \boldsymbol{a} = a_0a_1\dots a_k \dots 
$  
with  $ a_k \in \mathcal{A}_V $  and $ \Phi^{a_k} $ as in \eqref{phia}, define the corresponding $ V $-tuple  $ \boldsymbol{\omega}^{\boldsymbol{a} } $ of tree codes by
\begin{equation}\label{nseqcode}
\boldsymbol{\omega}^{\boldsymbol{a} }  =
(\omega_1^{\boldsymbol{a} } ,\dots,\omega_V^{\boldsymbol{a} } )  := \lim_{k \to \infty}  \Phi^{a_0} \circ  \Phi^{a_1} \circ  \dots \circ  \Phi^{a_{k}}(\omega^0_1,\dots,\omega^0_V), 
\end{equation}
for  any initial 
$ (\omega^0_1,\dots,\omega^0_V)\in \Omega^V$.

The sequence    $ \boldsymbol{a}  $   is called an \emph{address} for 
the $  V$-variable $ V $-tuple of tree codes $ \boldsymbol{\omega}^{\boldsymbol{a} } $.  

The set of all such addresses $ \boldsymbol{a}  $ is denoted by~   
$ \mathcal{A}_V^\infty $. 
\end{definition} 

Note that the tree code $ \Phi^{a_0} \circ  \Phi^{a_1} \circ  \dots \circ  \Phi^{a_{k }}(\omega^0_1,\dots,\omega^0_V)  $ is independent of $  (\omega^0_1,\dots,\omega^0_V)\in \Omega^V$ up to and including level $ k $, and hence agrees with $ \boldsymbol{\omega}^{\boldsymbol{a} } $ for these levels.

The sequence in \eqref{nseqcode} converges exponentially fast since   $ \Lip \Phi^{a_{k}}$ is $ \leq 1/M $.

The  map  $ \boldsymbol{a} \mapsto \boldsymbol{\omega}^{\boldsymbol{a} }:
\mathcal{A}_V^\infty \to \Omega^*_V $  is many-to-one, since the composition  of different $ \Phi^a $s may  give the same map even in simple situations as the following example
 shows. 

\begin{example}[\emph{Non uniqueness  of addresses}]  The  map  $ \boldsymbol{a} \mapsto \boldsymbol{\omega}^{\boldsymbol{a} }:
\mathcal{A}_V^\infty \to \Omega^*_V $  is many-to-one, since the composition  of different $ \Phi^a $s may  give the same map even in simple situations.  For example, suppose 
\[
M=1,\ V=2,\  F \in \Lambda,\
\Phi^a=   \begin{bmatrix}       
      F & 1 \\
     F & 1 \\
   \end{bmatrix},\
\Phi^b=   \begin{bmatrix}       
      F & 1 \\
      F & 2 \\
   \end{bmatrix}.
\]    
Since $ M=1 $  tree codes here are infinite sequences, i.e.\ 1-branching tree codes.
One readily checks from \eqref{phia} that
\[
\Phi^a(\omega,\omega') = (F*\omega, F*\omega), \ 
\Phi^b(\omega,\omega') = (F*\omega, F*\omega'),\
\]
and so
\begin{align*} 
\Phi^a\circ \Phi ^b (\omega,\omega')
&=  \Phi^a\circ \Phi ^a (\omega,\omega')
=  \Phi^b\circ \Phi ^a (\omega,\omega')
= (F*F*\omega,F*F*\omega), \\
  \Phi^b\circ \Phi ^b (\omega,\omega') &= (F*F*\omega,F*F*\omega').
\end{align*} 
\end{example}

The following definition is best understood from Example~\ref{rem17}. 
\begin{definition}[\emph{Tree skeletons}]\label{dfskel}
Given an address $ \boldsymbol{a} = a_0a_1\dots a_k \ldots 
\in \mathcal{A}_V^\infty $ and $ v \in \{1,\dots, V\} $ the corresponding \emph{tree skeleton} 
$ \widehat{J}^{\boldsymbol{a}}_v:T\to \{1,\dots,V\} $
is defined  by
\begin{equation} \label{tcc} 
\begin{gathered}
 \widehat{J}^{\boldsymbol{a}}_v(\emptyset)  = v,\quad  
 \widehat{J}^{\boldsymbol{a}}_v(m_1)  = J^{a_0}(v,m_1),\quad 
 \widehat{J}^{\boldsymbol{a}}_v(m_1m_2) = J^{a_1}( \widehat{J}^{\boldsymbol{a}}_v(m_1),m_2),\quad 
 \cdots \  ,\\
 \widehat{J}^{\boldsymbol{a}}_v(m_1\dots m_k)  = J^{a_{k-1}}( \widehat{J}^{\boldsymbol{a}}_v(m_1\dots m_{k-1}),m_{k}),\quad   \dots \  ,
\end{gathered}
\end{equation}
where the maps $ J^{a_k}(v,m) $ are as in \eqref{nIJ}.
\end{definition}

The tree skeleton depends on the maps $ J^{a_k}(w,m) $, but not on the maps $ I^{a_k}(w,m) $ and hence not on  the set  $ \{F^\lambda \}_{\lambda \in \Lambda } $ of IFSs and its indexing set $ \Lambda $.

The $V $-tuple of tree codes $ (\omega_1^{\boldsymbol{a} } ,\dots,\omega_V^{\boldsymbol{a} } ) $ can be recovered from  
the address $ \boldsymbol{a} = a_0a_1\dots a_k \ldots  $ as follows.

\begin{proposition}[\emph{Tree codes from addresses}]\label{adct}
If  $ \boldsymbol{a} = a_0a_1\dots a_k \ldots 
\in \mathcal{A}_V^\infty $ is an address, $I^{a_k} $ and $J^{a_k} $ are as in  \eqref{nIJ},    and $ \widehat{J}^{\boldsymbol{a}}_v $
is the tree skeleton corresponding to the $J^{a_k} $, then for each 
 $ \sigma \in T  $ and $ 1 \leq v \leq V $, 
\begin{equation} \label{wI}
\omega^{\boldsymbol{a}}_v(\sigma ) =I^{a_k}( \widehat{J}^{\boldsymbol{a}}_v(\sigma ) ) \text{ where }  k=|\sigma| .
\end{equation}
\end{proposition}

\begin{proof} The proof is implicit in Example~\ref{rem17}.  A formal proof can be given by induction.
\end{proof}

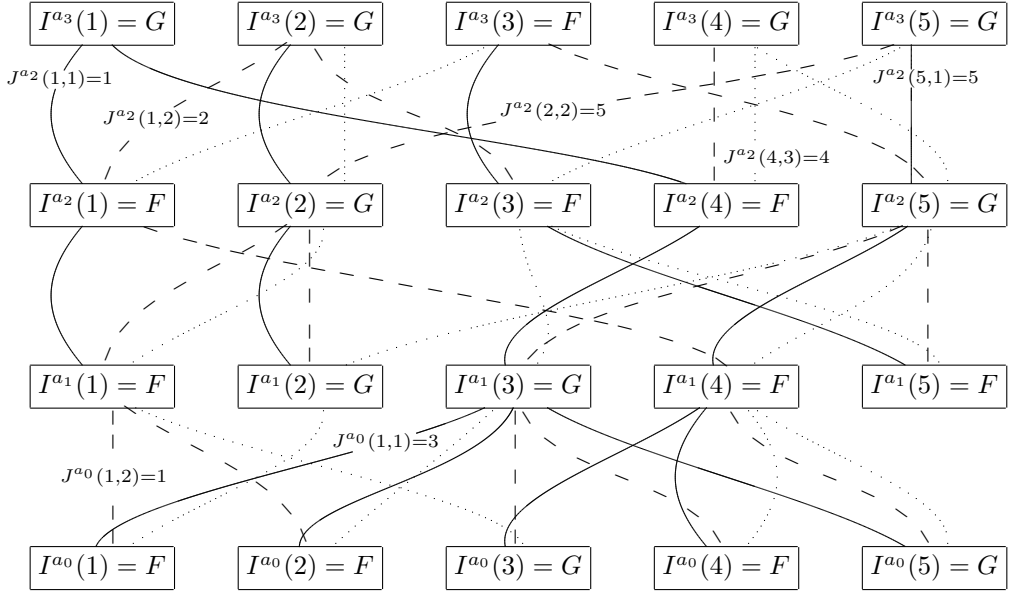
\begin{figure}[ptbh] 
   \centering
\begin{equation*}
\xymatrix@R+1cm{
*+[F]{ I^{a_3}(1)=G}   
  & *+[F]{ I^{a_3}(2)=G} 
  & *+[F]{I^{a_3}(3)=F} 
  & *+[F]{I^{a_3}(4)=G} 
  & *+[F]{I^{a_3}(5)=G}            \\
*+[F]{ I^{a_2}(1)=F  }        \ar@(ul,dl) @{-} [u] |(.7){J^{a_2}(1,1)=1}    
                                        \ar@(u,dl) @{--} [ur] |(.4){J^{a_2}(1,2)=2}      
                                        \ar@(ur,dl) @{.} [urr]              
  &  *+[F]{ I^{a_2}(2)=G}   \ar@(ul,dl) @{-} [u]     \ar@(u,dl) @{--} [urrr]   |(.45){J^{a_2}(2,2)=5}  
                                        \ar@(u,d) @<-3ex> @{.} [u]         
  &  *+[F]{ I^{a_2}(3)=F}   \ar@(ul,dl) @{-} [u]     \ar@(u,d) @{--} [ul]       \ar@(ur,dl)  @{.} [urr]                                      
  &  *+[F]{ I^{a_2}(4)=F}   \ar@(ul,d) @{-} [ulll]   \ar@(u,d) @<1ex> @{--} [u]       
                                        \ar @(u,d) @<-2.5ex> @{.} [u]  |(.2){\mbox{\phantom{xxx}}J^{a_2}(4,3)=4}                                 
  &  *+[F]{ I^{a_2}(5)=G}   \ar@(u,d) @<2ex> @{-} [u]   |(.75){\mbox{\phantom{xx}}J^{a_2}(5,1)=5}      
                                        \ar@(u,dr) @{--} [ull]       
                                         \ar @(ur,dr) @{.} [ul]   \\                    
*+[F]{ I^{a_1}(1) =F }        \ar@(ul,dl) @{-} [u]     \ar@(u,dl) @{--} [ur]        \ar@(ur,dr) @{.} [ur]                
  &  *+[F]{ I^{a_1}(2)=G}   \ar@(ul,dl) @{-} [u]     \ar@(u,d) @{--} [u]    \ar@(ur,dl)  @{.} [urrr]         
  &  *+[F]{ I^{a_1}(3)=G}   \ar@(ul,dl) @{-} [ur]     \ar@(u,dl) @{--} [urr]       \ar@(ur,d)  @{.} [u]                                      
  &  *+[F]{ I^{a_1}(4)=F}   \ar@(ul,dl) @{-} [ur]   \ar@(ur,dr) @{--} [ulll]        \ar @(ur,d)  @{.} [ur]                                    
  &  *+[F]{ I^{a_1}(5)=F}   \ar@(ul,dr)  @{-} [ull]     \ar@(u,d) @<.5 ex>  @{--} [u]    
                                        \ar @(ur,dr) @{.} [ull]    \\                    
*+[F]{ I^{a_0}(1)=F  }        \ar@(ul,dl) @{-} [urr]   |(.7){J^{a_0}(1,1)=3}   
                                        \ar@(u,d) @<-1ex> @{--} [u]    |(.5){J^{a_0}(1,2)=1}   
                                        \ar@(ur,dr) @{.} [ur]                
  &  *+[F]{ I^{a_0}(2)=F}   \ar@(ul,d) @{-} [ur]     \ar@(u,dr) @{--} [ul]          \ar@(ur,dl)  @{.}  [ur]         
  &  *+[F]{ I^{a_0}(3)=G}   \ar@(ul,dl) @{-} [ur]     \ar@(u,d) @<.2 ex> @{--} [u]       
                                        \ar@(ur,dr)  @{.} [ull]                                      
  &  *+[F]{ I^{a_0}(4)=F}   \ar@(ul,dl) @{-} [u]   \ar@(u,d) @{--} [ul]        \ar @(ur,dr)  @{.} [u]                                    
  &  *+[F]{ I^{a_0}(5)=G}   \ar@(ul,dr)  @{-} [ull]       \ar@(u,d) @{--} [ul]       
                                        \ar @(ur,dr) @{.} [ul]    \\                    
}
\end{equation*}
   \caption{Constructing tree codes  from an address.}
     \label{fig:buffer}
\end{figure}

\begin{example}[\emph{The Espalier\footnote{espalier [verb]: to train a fruit tree or ornamental shrub to grow flat against a wall, supported on a lattice.} technique}]\label{rem17}
We use this to find tree codes from addresses and addresses from tree codes.

We  first show how to represent an  address $ \boldsymbol{a}= a_0a_1\dots a_k \ldots  \in \mathcal{A}_V^\infty $ by means of a diagram as in Figure~\ref{fig:buffer}.    
From this we  construct the $ V $-variable $ V $-tuple   of tree codes  $ (\omega_1,\dots,\omega_V) \in \Omega^*_V $ with address $ \boldsymbol{a} $.  

Conversely, given a $ V $-variable $ V $-tuple of  tree codes $(\omega_1,\dots,\omega_V) \in  \Omega_V^* $ we show how to find the set of all  its possible  addresses.  
Moreover, given a single tree code $ \omega  \in \Omega_V $ we find all possible $(\omega_1,\dots,\omega_V) \in  \Omega_V^* $ with $ \omega_1 =\omega $  and all possible addresses in this case.

\smallskip
For the example here let  $ \Lambda = \{F,G\}  $ where $ F $ and $ G $ are symbols.  Suppose  $ M = 3 $ and $ V = 5 $.

Suppose    $ \boldsymbol{a} = a_0a_1\dots a_k \ldots \in \mathcal{A}_V^\infty $ is an address of $ (\omega_1,\dots,\omega_V) \in \Omega_V^* $, where
\begin{equation} \label{} 
a_0 = \begin{bmatrix} F & 3 & 1 &2 \\
    F & 3 & 1 &3 \\
    G & 4 & 3 &1 \\  
     F & 4 & 3 &4 \\
    G & 3 & 4 &4    
         \end{bmatrix},\
a_1 = \begin{bmatrix} F & 1 & 2 &2 \\
    G & 2 & 2 &5 \\
    G & 4 & 5 &3 \\  
     F & 5 & 1 &5 \\
    F & 3 & 5 &3   
         \end{bmatrix},\
a_2 = \begin{bmatrix} F & 1 & 2 &3 \\
    G & 2 & 5 &2 \\
    F & 3 & 2 &5 \\  
     F & 1 & 4 &4 \\
    G & 5 & 3 &4   
         \end{bmatrix},\
a_3 = \begin{bmatrix} G & * & * &* \\
   G & * & * &* \\
    F & * & * &* \\  
     G & * & * &* \\
    G & * & * &* \\    
         \end{bmatrix}.
                \end{equation} 
We will see that up to level 2 the tree codes $ \omega_1 $ and $ \omega_2 $ are those shown in Figure~\ref{fig:example}.  Although $ \omega_1 $ and $ \omega_2 $  are 3-variable and 2-variable respectively up to level 2,  it will follow from  Figure~\ref{fig:buffer} that they are  5-variable up to level~3 and are not 4-variable.                 
                 
The diagram  in Figure~\ref{fig:buffer} is obtained from    $ \boldsymbol{a} = a_0a_1a_2a_3\dots $   
by espaliering $ V $ copies of the tree $ T $ in Definition~\ref{dftrees}  up through an infinite lattice of $ V $ boxes at each level  $ 0,1,2,3, \dots $\,. One tree grows out of each box at level 0, and   one element  from $ \Lambda $ is assigned to each  box at each level.  When two or more branches pass through the same box from below they inosculate, i.e.\  their sub branches merge and are indistinguishable from that point upwards.
More precisely,   $ \boldsymbol{a}   $ determines the diagram in the following manner.   For each level $ k $ and starting from each box $ v $ at that level,                 
a branch $ \xymatrix@1@C=1.0cm{{}  \ar @{-}     [r]  & {} }$   terminates in  box number $  {J^{a_k}(v,1)}$ at level $ k+ 1 $, a   branch  $ \xymatrix@1@C=1.0cm{{}  \ar @{--}     [r]  & {} }$  terminates in box $  {J^{a_k}(v,2)}$ at level $ k+ 1 $ and a  branch $ \xymatrix@1@C=1.0cm{{}  \ar @{.}     [r]  & {} }$   terminates in  box $  {J^{a_k}(v,3)}$ at level $ k+ 1 $.  The element   $I^{a_k}(v) \in \Lambda $  is assigned to box $ v $ at level~$ k $.

Conversely, any such diagram determines a unique address $ \boldsymbol{a} = a_0a_1a_2a_3\dots\, $.
More precisely, consider an infinite lattice of $ V $ boxes at each level  $ 0,1,2,3, \dots $\,.  Suppose at each level $ k $ there is either $ F $ or $ G $ in each of the $ V $ boxes, and from each box there are 3 branches  $ \xymatrix@1@C=1.0cm{{}  \ar @{-}     [r]  & {} }$, $ \xymatrix@1@C=1.0cm{{}  \ar @{--}     [r]  & {} }$  and
$ \xymatrix@1@C=1.0cm{{}  \ar @{.}     [r]  & {} }$ , each branch terminating in a box  at level $ k+1$.   From this information one can read off $ I^{a_k}(v) $ and  $ J^{a_k}(v,m) $ for each $ k \geq 0 $, $ 1\leq v \leq V $ and $ 1 \leq m \leq M $, and hence determine  $ \boldsymbol{a}  $.

\smallskip
The diagram, and hence the address $ \boldsymbol{a} $, determines  the 
tree skeleton $ \widehat{J}^{\boldsymbol{a}}_v $ by assigning to each node of the copy of 
$ T $ growing out of box $ v $ at level $ 0 $,  the number of the particular box in which that node sits.
 If $ \boldsymbol{\omega}^{\boldsymbol{a} } = (\omega_1,\dots, \omega_V) $ is the $ V $-tuple of tree codes with address $ \boldsymbol{a} $ then the tree code $ \omega_v$ is obtained  by assigning to each node of the copy of $ T $ growing out of box $ v $ at level $ 0 $  the  element (fruit?) from $ \Lambda $ in the particular box in which that node sits.

\smallskip  
Conversely, suppose $ \omega \in \Omega_V  $ is a single  $ V $-variable tree code.  Then the set of all 
possible $ \boldsymbol{\omega} \in \Omega_V^* \subset \Omega^V $ of the form $ \boldsymbol{\omega} = (\omega_1, \omega_2, \dots , \omega_V ) $ with $ \omega_1 = \omega $, and the set of all diagrams and corresponding addresses $ \boldsymbol{a} \in \mathcal{A}_V^\infty $ for such $ \boldsymbol{\omega}$, is found as follows.

Espalier a copy of $T $ up through the infinite lattice with $ V $  boxes at each level  in such a way that if $ \sigma_0\dots \sigma_k $ and $  \sigma'_0\dots \sigma'_k $ sit in the same box at level $ k $ then the sub tree codes $ \omega\lfloor \sigma_0\dots \sigma_k $ and $ \omega\lfloor \sigma'_0\dots \sigma'_k $ are equal. Since $\omega $ is $ V $-variable, this  is always possible.  From level $ k $ onwards the two sub trees 
are fused together.

The possible  diagrams corresponding to this espaliered $ T $ are constructed as follows. For each $ \sigma \in T $ the element $ \omega(\sigma) \in \Lambda $ is assigned to the  box containing $ \sigma $.  By construction, this is the same element for any two $ \sigma $'s in the same box.  The three branches of the diagram from this box up to the next level are   given by the three sub branches of the espaliered $ T $ growing out of that box.  If $ T $ does not pass through some box, then the $ F $ or $ G $ in that box, and the three branches of the diagram from that box to the next level up, can be assigned arbitrarily. 

In this manner one obtains all possible diagrams for which the  tree growing out of box 1 at level 0 is $ \omega $.  Each diagram gives an address $ \boldsymbol{a} \in \mathcal{A}_V^\infty $ as before, and the corresponding   $ \boldsymbol{\omega}^{\boldsymbol{a}} = (\omega_1,\dots,\omega_V ) \in \Omega_V^* $ with address $ \boldsymbol{a}$ satisfies $ \omega_1 = \omega $.

In a similar manner, the set of possible  diagrams and corresponding addresses  can be obtained  for any   $ \boldsymbol{\omega} =  (\omega_1, \omega_2, \dots , \omega_V ) \in \Omega_V^* \subset \Omega^V $.
 \end{example}

\subsection{The Probability Distribution on $ V $-Variable  Tree Codes}\label{secpdtc}
Corresponding to the probability distribution $ P $  on $  \Lambda $ in~\eqref{dfF2} there are  natural probability distributions  $ \rho_V $ on $ \Omega_V $, $ \mathfrak{K}_V $ on $ \mathcal{K}_V $ and $ \mathfrak{M}_V $ on  $ \mathcal{M}_V $.  See Definition~\ref{Vvar} for notation  and for the following also note Definition~\ref{df34}.

\begin{definition}[\emph{Probability distributions on addresses and tree codes}]\label{df38}
The probability distribution 
$ P_V $ on $ \mathcal{A}_ V $, with notation as in \eqref{nIJ}, is defined by selecting $ a = (I,J) \in \mathcal{A}_V $ so that  $ I(1),\dots, I(V) \in \Lambda $ are iid with distribution $  P $, so that $J(1,1),\dots, J(V,M) \in  \{1,\dots, V\} $ are iid with the uniform distribution $ \{V^{-1},\dots , V^{-1} \} $, and so the $ I(v)  $ and $ J(w,m) $  are  independent of one another.

The probability distribution $ P_V^\infty $ on  $ \mathcal{A}^\infty_V $, the set  of addresses $ \boldsymbol{a} = a_0 a_1  \ldots  $, is defined by choosing the $ a_k $ to be iid with distribution $ P_V $.  

The probability distribution $ \rho_V^* $ on  $ \Omega_V^* $ is the image of $ P_V^\infty $ under the map $ \boldsymbol{a} \mapsto \boldsymbol{\omega}^{\boldsymbol{a}}$ in \eqref{nseqcode}.  
 
 The probability distribution $ \rho_V $ on $ \Omega_V $ is    the projection of $ \rho_V^* $ in any of the $ V $ coordinate directions. (By symmetry of the construction this is independent of choice of direction.) 
 \end{definition}
 
One obtains natural probability distributions on fractals sets and measures, and on $ V $-tuples of fractal sets and measures as follows. 
 \begin{definition}[\emph{Probability distributions on $ V $-variable fractals}] \label{df311}
 Suppose  $ (F^\lambda)_{\lambda \in \Lambda} $ is a uniformly contractive family of IFSs.  
 
 The probability distributions $ \mathfrak{K}^*_V $ and $ \mathfrak{K}_V $ on $ \mathcal{K}^*_V  $ and $\mathcal{K}_V  $ respectively are those induced from  $ \rho_V^* $  and $  \rho_V $ by the maps $ (\omega_1,\dots,\omega_V)\mapsto (K^{\omega_1},\dots, K^{\omega_V} ) $ and $ \omega \mapsto K^\omega $ in Definitions~\ref{df34} and~\ref{Vvar}.
 
Similarly, the  probability distributions $ \mathfrak{M}^*_V $ and $ \mathfrak{M}_V $  on $ \mathcal{M}^*_V  $ and $\mathcal{M}_V  $ respectively are those induced from $ \rho_V^* $ and $ \rho_V$ by the maps $ (\omega_1,\dots,\omega_V) \mapsto (\mu^{\omega_1},\dots, \mu^{\omega_V} ) $ and $ \omega \mapsto \mu^\omega $.
  \end{definition}
  
  That is, $ \mathfrak{K}^*_V $, $ \mathfrak{K}_V $, $ \mathfrak{M}^*_V $ and $ \mathfrak{M}_V $ are the probability distributions of the random objects $ (K^{\omega_1}, \dots, K^{\omega_V}) $, $ K^\omega$, $ (\mu^{\omega_1}, \dots, \mu^{\omega_V}) $ and $  \mu^\omega$ respectively, under the probability distributions $ \rho^*_V $ and $ \rho_V $ on $ (\omega_1,\dots,\omega_V ) $ and  $ \omega $.
Since  the projection of $ \rho_V^* $ in each   coordinate direction  is   $ \rho_V $ it follows that the  projection of $ \mathfrak{K}^*_V  $ in each coordinate direction is $ \mathfrak{K}_V $
and the  projection of $ \mathfrak{M}^*_V  $ in each coordinate direction is $ \mathfrak{M}_V $.
However,   there is a high degree of  dependence between the components and in general $  \rho_V^* \neq {\rho_V}^V$, $ \mathfrak{K}^*_V \neq  {\mathfrak{K}_V}^V $
and $ \mathfrak{M}^*_V \neq  {\mathfrak{M}_V}^V $.

  \begin{definition}[(\emph{The IFS acting on the set of $  V $-tuples of tree codes})] \label{dfIFStrp}
The IFS $ \boldsymbol{\Phi}_V $ in \eqref{trifs} is extended to an IFS  with probabilities by  
\begin{equation} \label{} 
 \boldsymbol{\Phi}_V:=(\Omega^V, \Phi^a, a \in \mathcal{A}_V, P_V ) .
\end{equation} 
\end{definition}

\begin{theorem}\label{macs}
A unique measure attractor  exists for  $ \boldsymbol{\Phi}_V $  and equals  $ \rho_V^* $.  In particular, the projection of  $ \rho_V^* $ in any coordinate direction is~$ \rho_V $.
\end{theorem} 

\begin{proof}
For $ a \in \mathcal{A}_V $   let $ R^a:\mathcal{A}_V^\infty \to \mathcal{A}_V^\infty  $ denote the  operator  $ \boldsymbol{a}\mapsto a*\boldsymbol{a} $.  Then 
$ (\mathcal{A}_V^\infty, R^a, a\in \mathcal{A}_V, P_V) $ is an IFS and the $ R^a  $ are contractive with Lipschitz constant $ 1/2 $ under  the metric $ d(\boldsymbol{a}, \boldsymbol{b}) = 2^{-k} $, where $ k  $ is the least integer such that $ a_k \neq b_k $.  Thus this IFS has a unique attractor   which from  Definition~\ref{df38} is $ P_V^\infty$.

Since each $ \Phi^a :\Omega^V \to \Omega^V$ has Lipschitz constant $ 1/M  $ from Theorem~\ref{char2},  it follows that $ \Phi ^a $  has Lipschitz constant $ 1/M  $ in the strong Prokhorov (and Monge-Kantorovitch) metric as a map on measures.  It also follows that $ \boldsymbol{\Phi}_V $ has a unique attractor from Theorem~\ref{char2} and Remark~\ref{nosep}.

Finally, if $ \Pi $ is the projection $ \boldsymbol{a} \mapsto \boldsymbol{\omega}^{\boldsymbol{a}}
: \mathcal{A}_V^\infty \to \Omega^V $  in \eqref{nseqcode}, it is immediate that $\Pi\circ R^a = \Phi^a \circ \Pi  $  and hence the attractor of $ \boldsymbol{\Phi}_V  $ is $ \Pi(P_V^\infty) =  \rho_V^*  $ by Definition~\ref{df38}.

The projection of  $ \rho_V^* $ in any coordinate direction is~$ \rho_V $ from Definition~\ref{df38}.
\end{proof}

\begin{remark}[\emph{Connection with other types of random fractals}]\label{compare}
The probability distribution $ \rho_V $ on $ \Omega_V $ is  obtained by projection from the probability distribution $ P^\infty_V $ on $ \mathcal{A}_V^\infty $, which  is constructed in a simple iid manner.  However, because of the combinatorial nature of the many-to-one map 
\[
 \boldsymbol{a} \mapsto \boldsymbol{\omega}^{\boldsymbol{a}}\mapsto \omega_1^{\boldsymbol{a}}
 :\mathcal{A}_V^\infty \to \Omega_V^*\to \Omega_V \text{ in \eqref{nseqcode},  inducing }
 P_V^\infty \to \rho_V^* \to \rho_V, 
 \]
the distribution    $ \rho_V $ is very difficult to analyse in terms of tree codes.  In particular, under the distribution $ \rho_V $ on $ \Omega_V $ and hence on $ \Omega$,  the set of random IFSs  $\omega \mapsto  F^{\omega(\sigma) }$ for  $ \sigma \in T$ has a complicated long range dependence structure.  (See the comments following Definition~\ref{dftrees}  for the notation $ F^{\omega(\sigma) }$.)

For each $ V $-variable tree code there are at most $ V $ isomorphism classes of subtree codes at each level, but the isomorphism classes are level dependent.  

Moreover, each set  of realisations of a $ V $-variable random fractal, as well as its associated probability distribution, is the projection of the fractal attractor of a single \emph{deterministic} IFS operating on $ V $-tuples of sets or measures.  See Theorems~\ref{th41} and~\ref{th42}.

For these reasons  $ \rho_V $,  $ \mathfrak{K}_V $ and  $ \mathfrak{M}_V $ are very different from other notions of a random fractal distribution in the literature.  See also Remark~\ref{gdf}.
\end{remark}

\section{Convergence and Existence Results for SuperIFSs}\label{seccgvvf}\label{secavcon}

We continue with the assumption that  $ \boldsymbol{F} = \{X, F^\lambda, \lambda  \in\Lambda, P\}   $ is a family of IFSs as in \eqref{dfF1} and~\eqref{dfF2}.

The   set $ \mathcal{K}_V $ of $ V $-variable fractal sets and the set $ \mathcal{M}_V $ of $ V $-variable fractal measures from Definition~\ref{Vvar}, together with their natural  probability distribution $ \mathfrak{K}_ V $ and $ \mathfrak{M}_V $ from Definition~\ref{df311}, are obtained as the attractors of IFSs $ \mathfrak{F}_V^{\mathcal{C}} $, $ \mathfrak{F}_V^{\mathcal{M}_c} $ or
$ \mathfrak{F}_V^{\mathcal{M}_1} $  under suitable conditions, see Theorems~\ref{th41} and \ref{th42}.  The Markov chains corresponding to these IFSs provide   MCMC  algorithms, such as the ``chaos game'',  for generating samples  of  $V$-variable fractal sets and $V$-variable fractal measures  whose empirical distributions converge to the stationary distributions  $ \mathfrak{K}_V  $ and $ \mathfrak{M}_V  $ respectively.

\begin{definition}\label{dfextmet}
The metrics $ d_{\mathcal{H}}$, $ d_P $  and  $ d_{MK} $ are defined  on $ \mathcal{C} (X)^V $, $ \mathcal{M}_c(X)^V $ and  $ \mathcal{M}_1(X)^V $ by 
\begin{equation}\label{mets}
\begin{aligned}
d_{\mathcal{H}}\left( (K_1,\dots, K_V) , (K'_1,\dots, K'_V) \right)&=
         \max_v  d_{\mathcal{H} }(K_v,K'_v),\\
d_{P}\left( (\mu_1,\dots, \mu_V) , (\mu'_1,\dots, \mu'_V) \right)&=
         \max_v  d_{P}(\mu_v,\mu'_v),\\
d_{MK}\left( (\mu_1,\dots, \mu_V) , (\mu'_1,\dots, \mu'_V) \right)&=
      V^{-1} \sum_v  d_{MK}(\mu_v,\mu'_v),         
                  \end{aligned}
\end{equation}
where the metrics on the right are as in Section~\ref{xd}.
\end{definition}
 
 The  metrics $ d_{\mathcal{H}}$ and $ d_{MK} $ are complete and separable, while the metric $ d_P $ is complete but usually not separable. See Definitions~\ref{cx},  \ref{mx2} and \ref{mp}, the comments which follow them, Proposition~\ref{prop25} and Remark~\ref{nsp}. 
 The metric  $ d_{MK} $ is usually not locally compact, see Remark~\ref{rmloc}.
 
The following IFSs  \eqref{eq44} are  analogues of the tree IFS $\boldsymbol{\Phi}_V:=(\Omega^V, \Phi^a, a \in \mathcal{A}_V ) $ in~\eqref{trifs}.

\begin{definition}[\emph{SuperIFS}]\label{df42}
For $ a \in \mathcal{A}_V $ as in  Definition~\ref{dfIFStr}    let 
\[
\mathcal{F}^a :  \mathcal{C} (X)^V \to  \mathcal{C} (X)^V, \quad  
 \mathcal{F}^a :  \mathcal{M}_c (X)^V \to  \mathcal{M}_c (X)^V,\quad 
 \mathcal{F}^a :  \mathcal{M}_1 (X)^V \to  \mathcal{M}_1 (X)^V,
 \]
be given by
\begin{equation}\label{fcom}
\begin{aligned} 
\mathcal{F}^a (K_1,\dots,K_V)  
&=  \left( F^{I^a(1)}\bigl(K_{J^a(1,1)},\dots,K_{J^a(1,M)}\bigr), \dots ,
F^{I^a(V)}\big(K_{J^a(V,1)},\dots,K_{J^a(V,M)}\big)\right),\\
\mathcal{F}^a (\mu_1,\dots,\mu_V)  
&=  \left( F^{I^a(1)}\bigl(\mu_{J^a(1,1)},\dots,\mu_{J^a(1,M)}\bigr), \dots ,
F^{I^a(V)}\big(\mu_{J^a(V,1)},\dots,\mu_{J^a(V,M)}\big)\right),
\end{aligned} 
\end{equation} 
where  the action of $ F^{I^a(v)}  $ is defined in \eqref{eq21a}.

Let 
\begin{equation} \label{eq44}
\begin{aligned} 
 \mathfrak{F}_V^{\mathcal{C}} 
 &= \left(\mathcal{C} (X)^V,\mathcal{F}^a, a \in \mathcal{A}_V , P_V\right),  \\
\mathfrak{F}_V^{\mathcal{M}_c} 
&= \left(\mathcal{M}_c (X)^V,\mathcal{F}^a, a \in \mathcal{A}_V , P_V\right)  , \\
\mathfrak{F}_V^{\mathcal{M}_1} 
&= \left(\mathcal{M}_1 (X)^V,\mathcal{F}^a, a \in \mathcal{A}_V , P_V\right)  ,  
\end{aligned}  
\end{equation} 
be the corresponding IFSs, with  $ P_V $ from Definition~\ref{df38}. 
These IFSs  are called   \emph{superIFSs}.
\end{definition}

Two types of conditions will be used on families of IFSs.  The first was introduced in~\eqref{equc}.

\begin{definition}[\emph{Contractivity conditions for a family of IFSs}]\label{auc}
The family      $\boldsymbol{F}= \{X, F^\lambda, \lambda  \in\Lambda, P\}   $ 
  of IFSs 
is \emph{uniformly contractive}  and  \emph{uniformly bounded} if for some   $0\leq  r<1$,
\begin{equation} \label{equc1} 
\sup\nolimits_\lambda \max\nolimits_{m} d(f_m^\lambda(x), f_m^\lambda(y)) \leq r \,d(x,y) 
\quad \text{and} \quad  \sup\nolimits_\lambda \max\nolimits_m  d(f^\lambda_m(a),a) <\infty
\end{equation} 
for all $ x,y\in X  $ and some $ a\in X $. 
(The probability distribution $ P  $ is not used in~\eqref{equc1}.  
The second condition is immediate if $ \Lambda $  is finite.)  

The family  $ \boldsymbol{F} $ 
is \emph{pointwise average contractive} and \emph{average bounded} 
if for some $ 0\leq r<1 $ 
\begin{equation} \label{eqac}
\expected_\lambda \expected_m 
       d(f^\lambda_m(x), f^\lambda_m(y)) \leq r d(x,y) 
    \quad \text{and} \quad
  \expected_\lambda \expected_m d(f^\lambda_m(a),a ) <\infty 
\end{equation} 
for all $ x,y \in X $  and some  $ a\in X $.
\end{definition}


The following theorem includes and strengthens   Theorems 15--24   from \cite{BHS05}.  
The space $ X $ may be noncompact and  the strong Prokhorov metric $ d_P $ is used rather than the Monge-Kantorovitch metric $ d_{MK}  $.  

\begin{theorem}[\emph{Uniformly contractive conditions}]  \label{th41} Let $ \boldsymbol{F} = \{X, F^\lambda, \lambda  \in\Lambda, P\}   $ 
be a \emph{finite} family of IFSs  on a complete separable metric space $ (  X,d)  $ satisfying~\eqref{equc1}.  

Then the superIFSs $  {\mathfrak{F}}_V^{\mathcal{C}} $ and $  {\mathfrak{F}}_V^{\mathcal{M}_c} $
satisfy the uniform contractive  condition  $ \Lip \mathcal{F}^a \leq r $.
Since $ (\mathcal{C} (X)^V, d_{\mathcal{H}}) $ is complete and separable,  and  $ (\mathcal{M}_c (X)^V, d_{P}) $ is complete but not necessarily separable,
  the corresponding conclusions of Theorem~\ref{th1} and Remark~\ref{nosep} are valid.

In particular
\begin{enumerate}

\item $ \mathfrak{F}_V^{\mathcal{C}} $ and $ \mathfrak{F}_V^{\mathcal{M}_c} $ each have  unique  compact  set attractors and compactly supported  separable measure attractors.  The  attractors are  $  \mathcal{K}^*_V $  and  $ \mathfrak{K}^*_V$,  and
 $ \mathcal{M}^*_V $  and  $ \mathfrak{M}^*_V$, respectively.  Their projections   in any coordinate direction are $ \mathcal{K}_V $,  $ \mathfrak{K}_V$,  
 $ \mathcal{M}_V $  and  $ \mathfrak{M}_V$,  respectively.
 
\item The Markov chains generated by the superIFSs converge  at an  exponential rate.  

\item 
Suppose $ (K^0_1, \dots, K^0_V) \in \mathcal{C}(X)^V $.  If $ \boldsymbol{a} = a_0a_1\dots \in \mathcal{A}_V^\infty $ then for some $ (K_1, \dots, K_V)\in \mathcal{K}_V^* $ which is  independent of $ (K^0_1, \dots, K^0_V)$, 
\[
\mathcal{F}^{a_0}\circ \dots \circ \mathcal{F}^{a_k}
     (K^0_1, \dots, K^0_V)  \to (K_1, \dots, K_V) 
     \]
in $  (\mathcal{C}(X)^V, d_\mathcal{H} )$ as $ k \to \infty $. Moreover
\[     
\left\{\mathcal{F}^{a_0}\circ \dots \circ \mathcal{F}^{a_k}(K^0_1,\dots K^0_V) : 
                   a_0 ,\dots, a_k  \in   \mathcal{A}_V    \right\}  \to \mathcal{K}_V^*
\]
in     $  \big(\mathcal{ C}\big(\mathcal{C}(X)^V, d_\mathcal{H} \big), d_\mathcal{H} \big)$.  Convergence is  exponential  in both cases.

\quad Analogous results apply for  $ \mathcal{M}_V^* $, $ \mathfrak{K}_ V^* $ and $ \mathfrak{M}_V^* $, with $ d_P $ or $ d_\mathcal{H} $ as appropriate.

\item Suppose $ \boldsymbol{B}^0=(B_{1}^0,\dots,B_{V}^0)\in \mathcal{C}(X)^V $ and  $ \boldsymbol{a} = a_0a_1\dots \in \mathcal{A}_V^\infty $  and let 
$ \boldsymbol{B}^k(\boldsymbol{a})  =  \boldsymbol{B}^k =
\mathcal{F}^{a_k}(\boldsymbol{B}^{k-1} )$  if $  k \geq 1 
$.  Let $ B^k(\boldsymbol{a}) $ be the first component of $ \boldsymbol{B}^{k}(\boldsymbol{a})$. 
Then for a.e.\ $ \boldsymbol{a}  $ and every $ \boldsymbol{B}^0 $, 
\begin{equation} \label{eq45}
\frac{1}{k} \sum\nolimits_{n=0}^{k-1} \delta_{  {B}^n(\boldsymbol{a})}
\to \mathfrak{K}_V 
\end{equation} 
weakly in the space of probability distributions on $ (\mathcal{C} (X), d_{\mathcal{H} } )$.

\quad  For  starting measures $ (\mu_{1}^0,\dots,\mu_{V}^0)\in \mathcal{M}_c(X)^V $,  
 there are analogous results modified as in Remark~\ref{nosep} to account for the fact that $ (\mathcal{M}_c(X) , d_P )$ is not separable.

\quad There are similar results for $ V $-tuples of sets or measures.

\end{enumerate} 
\end{theorem}

\begin{proof}    The assertion  $ \Lip \mathcal{F}^a \leq r $  follows from \eqref{mets}  by a straightforward  argument  using Definitions~\ref{cx} and \ref{mp}, equation \eqref{spl} and the comment following it, and equations \eqref{IFSro}, \eqref{mets} and \eqref{fcom}.  So the analogue of the uniform  contractive condition \eqref{equcb}  in Theorem~\ref{th1} is satisfied, while the uniform boundedness condition is immediate since $\boldsymbol{F} $ is finite.  
From Theorem~\ref{th1}.d and Remark~\ref{nosep},  $ \mathfrak{F}_V^{\mathcal{C}} $ and $ \mathfrak{F}_V^{\mathcal{M}_c} $ each has a unique set attractor which is a subset of $ \mathcal{C} (X)^V $ and $ \mathcal{M}_c (X)^V $ respectively, and a measure attractor which is a probability distribution on the respective set attractor.  

It remains to identify the attractors with the sets and distributions  in Definitions~\ref{Vvar}, \ref{df34},  \ref{df38} and \ref{df311}. 
Let  $ \widehat{\Pi}$ denote one of the  maps
\begin{equation} \label{wkmap} 
 \omega \mapsto K^\omega ,\
 \omega \mapsto \mu^\omega ,\
(\omega_1, \dots, \omega_V ) \mapsto (K^{\omega_1},\dots, K^{\omega_V}),\
  (\omega_1, \dots, \omega_V ) \mapsto (\mu^{\omega_1},\dots, \mu^{\omega_V})  ,
\end{equation} 
depending on the context. In the last two  cases  it follows from \eqref{phia} and \eqref{fcom} that
$
\widehat{\Pi}\circ \Phi^a = \mathcal{F}^a \circ \widehat{\Pi} 
$.
Also denote by  $ \widehat{\Pi} $ the extension of $ \widehat{\Pi} $   to a map on sets, on $ V $-tuples of sets,  on measures or  on $ V $-tuples of  measures, respectively. It follows from Theorem~\ref{th1}.d together with  Definitions~\ref{Vvar}, \ref{df34} and~\ref{df38} that
\begin{equation} \label{kpi}
\begin{alignedat}{4} 
\mathcal{K}^*_V &= \widehat{\Pi}(\Omega^*_V),&\ 
\mathcal{K}_V &= \widehat{\Pi}(\Omega_V),& \
\mathfrak{K}^*_V &= \widehat{\Pi}(\rho^*_V),&\    
\mathfrak{K}_V &= \widehat{\Pi}(\rho_V),\\
\mathcal{M}^*_V &= \widehat{\Pi}(\Omega^*_V),&\ 
\mathcal{M}_V &= \widehat{\Pi}(\Omega_V),& \
\mathfrak{M}^*_V &= \widehat{\Pi}(\rho^*_V),&\    
\mathfrak{M}_V &= \widehat{\Pi}(\rho_V).
\end{alignedat}
\end{equation} 
The rest of (i)  follows from Theorems~\ref{char2} and \ref{macs}.

The remaining parts  of  the theorem follows from Theorem~\ref{th1} and Remark~\ref{nosep}. 
\end{proof}

\begin{remark}[\emph{Why use the $ d_P $ metric?}]
For computing approximations to the set of $ V $-variable fractals and its associated probability distribution  which correspond  to $ \boldsymbol{F} $,    the main part of the theorem is (iv) with either sets or measures.  
The advantage of      $ (\mathcal{M}_c(X) , d_P )$ over $ (\mathcal{M}_c(X), d_{MK} )$ is that 
for use in the analogue of \eqref{tfc}  the space  $ \mathcal{B} \mathcal{C}    (\mathcal{M}_c(X) , d_P ) $ is much larger than $ \mathcal{B} \mathcal{C}    (\mathcal{M}_c(X) , d_{MK} ) $.
For example, if $\phi (\mu) = \psi\big( d_{P}(\mu, \mu_1)\big)$, where  $ \mu_1 \in \mathcal{M}_c(X)$   and $ \psi $ is a continuous cut-off  approximation to the characteristic function of   $ [0,\epsilon ] \subset \mathbb{R} $, then $ \phi \in  \mathcal{B} \mathcal{C}    (\mathcal{M}_c(X) , d_P ) $ but  is not 
continuous or even Borel over~$ (\mathcal{M}_c(X) , d_{MK} )$.
  \end{remark}


\begin{theorem}[\emph{Average contractive conditions}] \label{th42} Let $\boldsymbol{F} =  \{X, F^\lambda, \lambda  \in\Lambda, P\}   $ 
be a possibly infinite family of IFSs  on a complete separable metric space $ (  X,d)  $ satisfying~\eqref{eqac}.

Then the superIFS    $ \mathfrak{F}_V^{\mathcal{M}_1} $
 satisfies the pointwise average  contractive  and average boundedness conditions 
\begin{equation} \label{acab} 
\expected_a d_{MK}  \big(
        \mathcal{F}^a(\boldsymbol{\mu}),  \mathcal{F}^a(\boldsymbol{\mu}')\big)
\leq r d_{MK} (\boldsymbol{\mu} ,   \boldsymbol{\mu}'), \quad 
\expected_a d_{MK} \big(
        \mathcal{F}^a(\boldsymbol{\mu}^0),   \boldsymbol{\mu}^0)\big)
<\infty
\end{equation} 
for all $ \boldsymbol{\mu},   \boldsymbol{\mu}' \in \mathcal{M}_1(X)^V $ and some $  \boldsymbol{\mu}^0 \in 
  \mathcal{M}_1(X)^V $.
 Since $ (\mathcal{M}_1 (X)^V, d_{MK}) $ is complete and separable, the corresponding conclusions of Theorem~\ref{th1} are valid.

In particular
\begin{enumerate}

\item $ \mathfrak{F}_V^{\mathcal{M}_1} $ has a unique measure attractor and its  projection  in any coordinate direction is  the same.  The attractor and the projection  are denoted by  $ \mathfrak{M}^*_V$ and  $ \mathfrak{M}_V$, and extend the corresponding distributions in Theorem~\ref{th41}.

\item For a.e.\ $ \boldsymbol{a} = a_0a_1\dots \in \mathcal{A}_V^\infty $, if $ (\mu_1^0, \dots, \mu_V^0 ) \in \mathcal{M}_1(X)^V $ then 
$ \mathcal{F}^{a_0}\circ \dots \circ \mathcal{F}^{a_k}
     (\mu^0_1, \dots, \mu^0_V) 
   $  converges  at an exponential  rate.    The limit random $ V $-tuple of measures has probability distribution $ \mathfrak{M}_V^* $. 

\item If   $ \boldsymbol{a} = a_0a_1\dots \in \mathcal{A}_V^\infty $ and $ \boldsymbol{\mu}^0 =  (\mu_1^0, \dots, \mu_V^0 ) \in \mathcal{M}_1(X)^V $ let $ \boldsymbol{\mu}^k (\boldsymbol{a}) = \mathcal{F}^{a_k}(\boldsymbol{\mu}^{k-1} )$ for $ k \geq 1 $.  Let $ \mu^k (\boldsymbol{a}) $ be the first component of $ \boldsymbol{\mu}^k (\boldsymbol{a}) $. Then  
for a.e.\ $ \boldsymbol{a}  $ and every $ \boldsymbol{\mu}^0 $, 
\begin{equation} \label{eq450}
\frac{1}{k} \sum\nolimits_{n=0}^{k-1} \delta_{ \mu^n (\boldsymbol{a})}
\to \mathfrak{M}_V 
\end{equation} 
weakly in the space of probability distributions on $ (\mathcal{M}_1(X), d_{MK} )$.
\end{enumerate}
\end{theorem}

\begin{proof}   
To establish average boundedness   in \eqref{acab}  let $ \boldsymbol{\mu}^0= (\delta_{b},\dots, \delta_{b}) \in \mathcal{M}_1(X)^V $ for some $ b \in X $. Then
\begin{align*}
\expected_a d_{MK} &\bigl( \mathcal{F}^a(\delta_{b},\dots, \delta_{b}), (\delta_{b},\dots, \delta_{b}) \bigr)\\
 & = \expected_a \frac{1}{V} \sum_v d_{MK} \Bigl( \sum_m w_m^{I^a(v)}\delta_{f_m^{I^a(v)}({b})},\delta_{b} \Bigr) \quad \text{from \eqref{mets}, \eqref{fcom},  \eqref{dfF1} and \eqref{eq21a}}  \\
 &\leq \expected_a \frac{1}{V} \sum_v\sum_m w_m^{I^a(v)} d_{MK}(\delta_{f_m^{I^a(v)}({b})},\delta_{b})
 \quad \text{by basic properties of $ d_{MK}$} \\ 
 &=\expected_a \frac{1}{V} \sum_v\sum_m w_m^{I^a(v)} d(f_m^{I^a(v)}({b}),{b}) 
  \quad \text{basic properties of $ d_{MK}$}\\ 
&= \expected_\lambda \frac{1}{V} \sum_v \sum_m   w_m^\lambda d(f_m^{\lambda}({b}),{b})  
\quad \text{since $ \dist  I^a(v) =P= \dist \lambda $ by  Definition~\ref{df38}}\\
&=\expected_\lambda \expected_m  d(f^\lambda_m({b}),{b}) 
 < \infty .
\end{align*}

To establish average contractivity   in \eqref{acab} let 
 $ (\mu_1,\dots,\mu_V), (\mu'_1,\dots,\mu'_V) \in   \mathcal{M}_1(X)^V$.
 Then
{\allowdisplaybreaks
\begin{align*}
\expected_a &d_{MK}  \big(
        \mathcal{F}^a(\mu_1,\dots,\mu_V),  \mathcal{F}^a(\mu'_1,\dots,\mu'_V)
                                                   \big)  \\  
&\leq \expected_a \frac{1}{V} \sum_{v}   
                d_{MK}\Big(\sum_m w_m^{I^a(v)} f_m^{I^a(v)}( \mu_{J^a(v,m)}), 
                                     \sum_m w_m^{I^a(v)} f_m^{I^a(v)}( \mu'_{J^a(v,m)})\Big) 
  \intertext{\hfill from  \eqref{fcom},  \eqref{dfF1} and \eqref{eq21a}}
&\leq \expected_a \frac{1}{V} \sum_{v}    \sum_{m} 
      w_m^{I^a(v)}
                d_{MK}\big( f_m^{I^a(v)}( \mu_{J^a(v,m)}), 
                                       f_m^{I^a(v)}( \mu'_{J^a(v,m)})\big) \quad \text{by properties of $ d_{MK}$}\\
&= \expected_\lambda \frac{1}{V}\sum_v\sum_m
		w_m^\lambda \expected_a                                      
                d_{MK}\big(  f_m^{\lambda }( \mu_{J^a(v,m)}), 
                                        f_m^{\lambda }( \mu'_{J^a(v,m)})\big) 
 \intertext{\hfill by the independence of $ I^a(v) $  and $ J^a(v,n) $ in Definition~\ref{df38} and since $ \dist I^a(v) = P = \dist \lambda $ }
&= \expected_\lambda \frac{1}{V}\sum_v\sum_m
		w_m^\lambda    \expected_t                                  
                d_{MK}\big(f_m^{\lambda }( \mu_{t}), 
                                       f_m^{\lambda }( \mu'_{t})\big)
\intertext{by the uniform distribution of $ J^a(v,m) $ for fixed  $ (v,m) $ where $t $ is distributed uniformly  over  $   \{1,\dots, V\} $}                                       
&=   \expected_t
           \expected_\lambda \expected_m
                d_{MK}\big(f_m^{\lambda }( \mu_{t}), 
                                       f_m^{\lambda }( \mu'_{t})\big).
\end{align*}
}%

Next let $ W_t,W'_t $ be random variables on  $  X    $    such that $  \dist   W_t =\mu_t $,  $  \dist   W'_t =\mu'_t $ and     $ \expected d(W_t, W'_t) =  d_{MK}(\mu_t,\mu'_t)$,  where $ \expected $ without a subscript here and later refers to expectations from the sample space over which the  $ W_t  $ and $  W'_t  $ are jointly defined.  This is possible by \cite{Dud}*{Theorem~11.8.2}.    Then
{\allowdisplaybreaks
\begin{align}
\expected_t\expected_\lambda&\expected_m  
                d_{MK}\big(f_m^{\lambda }( \mu_{t}), 
                                       f_m^{\lambda }( \mu'_{t})\big) \notag\\
&\leq    \expected_t\expected_\lambda\expected_m \expected
                d \big(f_m^{\lambda }(W_{t}), 
                                       f_m^{\lambda }( W'_{t})\big)
                                       \quad \text{by the third version of \eqref{mk}}\notag\\
&\leq r\expected_t \expected  d(W_t,W'_t) \quad \text{from \eqref{eqac}}\notag\\
&= r \expected_t d_{MK}(\mu_t,\mu'_t) \quad \text{by choice of  $ W_t$ and $ W'_t $} \notag\\
&=  r  d_{MK}\big((\mu_1,\dots, \mu_V), (\mu'_1,\dots,\mu'_V)\big) .\notag
\end{align} 
}

This completes the proof of \eqref{acab}.  The remaining conclusions now follow  by Theorem~\ref{th1}.
\end{proof}

\begin{remark}[\emph{Global average contractivity}]\label{npc} 
One might expect that the global average contractive condition $ \expected_\lambda \expected_m 
    \Lip  f_m^\lambda  <1
$ on the family $ \boldsymbol{F} $ would imply  the global average contractive condition $ \expected_{a} \Lip \mathcal{F}^{a}<1 $,
i.e.\ would imply that the superIFS    $ \mathfrak{F}_V^{\mathcal{M}_1} $ is global average contractive.  However, this is not the case.

For example, let $ X = \mathbb{R} $, $ V=2 $ and $ M=2 $. Let $ \boldsymbol{F} $ contain  a single IFS $ F = (f_1,f_2; 1/2, 1/2 ) $ where 
\[
 f_1(x) =-\frac{3}{2}\, x, \quad f_2(x) = \frac{1+\epsilon}{2} + \frac{1-\epsilon}{2}\, x.
 \]
 Then $ \Lip  f_1 = 3/2$,   $ \Lip  f_2 = (1-\epsilon)/2$ and $ \expected_m \Lip  f_m = 1-\epsilon/4 $. So $ \boldsymbol{F} $ is global average contractive.
 
 Note $ f_1(0)=0 $ and $ f_2(1)=1 $.
  Let $ \boldsymbol{\mu} = (\delta_0, \delta_0)$, $  \boldsymbol{\mu}' = (\delta_1, \delta_1)$ and note   $ d_{MK}( \boldsymbol{\mu},  \boldsymbol{\mu}') = 1 $.  Then for   any $a \in \mathcal{A}_V $ as in~\eqref{nIJ}, 
 \[
 \mathcal{F}^{a}(\boldsymbol{\mu}) =\left(\frac{1}{2} \delta_0 + \frac{1}{2} \delta_{\frac{1+\epsilon}{2}} , 
                                                   \      \frac{1}{2} \delta_0 + \frac{1}{2} \delta_{\frac{1+\epsilon}{2}}\right), \quad 
  \mathcal{F}^{a} (\boldsymbol{\mu}') =\left(\frac{1}{2} \delta_{-\frac{3}{2}} + \frac{1}{2} \delta_1 , 
                                                     \       \frac{1}{2} \delta_{-\frac{3}{2}} + \frac{1}{2} \delta_1\right).
\] 
From the first form of   \eqref{mk} with   $ f(x) = |x| $ and using~\eqref{mets},
$ d_{MK} ( \mathcal{F}^{a} (\boldsymbol{\mu}),  \mathcal{F}^{a} (\boldsymbol{\mu}') )\geq 1-\epsilon/4 $ and so $ \Lip  \mathcal{F}^{a} \geq 1-\epsilon/4 $ for every $ a $.

Next let $ a^* =    \begin{bmatrix} 
      F & 1 & 2 \\
      F & 1 & 2 
   \end{bmatrix} $ and choose $ \boldsymbol{\mu} = (\delta_0, \delta_0)$, $  \boldsymbol{\mu}' = (\delta_1, \delta_0)$, so $ \boldsymbol{\mu} $ and  $  \boldsymbol{\mu}' $  differ in the box on which $ f_1$  always acts and agree in the box on which $ f_2$  always acts.  Note  $ d_{MK}( \boldsymbol{\mu},  \boldsymbol{\mu}') = 1/2 $. Then 
 \[
 \mathcal{F}^{a^*}(\boldsymbol{\mu}) =\Big(\frac{1}{2} \delta_0 + \frac{1}{2}\delta_{\frac{1+\epsilon}{2}}, 
                                                   \      \frac{1}{2} \delta_0 + \frac{1}{2} \delta_{\frac{1+\epsilon}{2}}\Big),\quad   
  \mathcal{F}^{a^*} (\boldsymbol{\mu}') =\Big(\frac{1}{2} \delta_{-\frac{3}{2}} + \frac{1}{2}  \delta_{\frac{1+\epsilon}{2}} , 
                                                   \      \frac{1}{2} \delta_{-\frac{3}{2}} +  \frac{1}{2} \delta_{\frac{1+\epsilon}{2}}\Big).
\] 
Again using the first form of   \eqref{mk} with $ f(x) = |x| $, it follows that
$ d_{MK} \big( \mathcal{F}^{a^*} (\boldsymbol{\mu}),  \mathcal{F}^{a^*} (\boldsymbol{\mu}')\big) \geq 3/4 $, so  $ \Lip  \mathcal{F}^{a^*} \geq 3/2 $.

Since there are 16 possible maps $ a \in \mathcal{A}_V $, each selected with probability $ 1/16 $, it follows that
\[
 \expected_{a} \Lip \mathcal{F}^{a} \geq \frac{15}{16} \left(1-\frac{\epsilon}{4}\right) + \frac{1}{16}\cdot \frac{3}{2} > 1 \quad  \text{if} \quad  \epsilon < \frac{2}{15}  .  
 \]

 So for such $0< \epsilon < 2/15  $ the IFS $ \mathfrak{F}_V^{\mathcal{M}_1} $ is not  global  average contractive.  
But since $ \expected_m \Lip  f_m = 1-\epsilon/4 $ it follows from Theorem~\ref{th42} that $ \mathfrak{F}_V^{\mathcal{M}_1} $ is pointwise average contractive, and so Theorem~\ref{th1} can be applied.
 \end{remark}

\begin{example}[\emph{Random curves in the plane}] \label{rm54} The following  shows why it is natural to consider  families of IFSs  which are both infinite and not uniformly contractive.
Such examples can be modified to model Brownian motion and other stochastic processes, see \cite{Graf91}*{Section 5.2} and \cite{HR98_1}*{pp 120--122}.

Let $\boldsymbol{F} = \{\mathbb{R}^2, F^\lambda, \lambda \in \mathbb{R}^2,  N(0,\sigma^2 I)  \} $ where $ F^\lambda = \{f^\lambda_1, f^\lambda_2;1/2,1/2\} $ and $ N(0,\sigma^2I) $ is 
the symmetric normal distribution   in $ \mathbb{R}^2 $ with variance $ \sigma^2 $. The functions $  f^\lambda_1$ and $ f^\lambda_2 $ are uniquely specified by the   requirements that they be similitudes with positive determinant and
\[
f^\lambda_1(-1,0)= (-1,0), \ f^\lambda_1(1,0)= \lambda, \
f^\lambda_2(-1,0)= \lambda, \ f^\lambda_2(1,0)= (1,0).
\]
A calculation shows $ \sigma=1.42  $ implies  $ \expected_\lambda \expected_m \Lip f^\lambda_m \approx 0.9969 $ and so average contractivity holds if  $ \sigma \leq 1.42 $.  If $  |\lambda| $ is sufficiently large then neither  $  f^\lambda_1$ nor $  f^\lambda_2 $ are contractive.

The IFS $ F^ \lambda $ can also be interpreted as a map from the space $  \mathcal{C}( [0,1] , \mathbb{R}^2)  $, of continuous paths from $ [0,1] $ to $ \mathbb{R}^2  $, into itself  as follows:
\[
(F^\lambda(\phi))(t) = \begin{cases}  f_1^\lambda(\phi(2t)) & 0\leq t \leq \frac{1}{2}, \\
                                                               f_2^\lambda(\phi(2t-1)) & \frac{1}{2} \leq t \leq 1.
                                    \end{cases}
\]
Then one can define a superIFS acting on such functions in a manner analogous to that for the superIFS acting on sets or measures.  Under the average contractive condition one obtains $ L^1 $ convergence to a class of $ V $-variable fractal paths, and in particular $ V $-variable fractal curves, from $(-1,0) $ to $ (1,0) $.  We omit the details.
\end{example}

\begin{remark}[\emph{Graph directed fractals}]\label{gdf} 
Fractals generated by a graph directed system [GDS]  or more generally by a graph directed Markov system [GDMS], or by random versions of these, have been considered by many authors.  See \cites{Ols, MU} for the definitions and references.  We comment here on the connection between these fractals and $ V $-variable fractals.

In particular, for each experimental run of its generation process, a random GDS or GDMS generates a single realisation of the associated  random fractal.  On the other hand, for each run, a superIFS generates a \emph{family} of realisations whose empirical distributions  converge  a.s.\ to the probability distribution given by the associated $ V $-variable  random fractal.

\smallskip To make a more careful comparison, allow the number of functions $ M $ in each IFS $ F^\lambda $ as in \eqref{dfF1} to depend on $ \lambda $.  A GDS   together with its associated contraction maps  can be interpreted as a map from $ V $-tuples of sets to $ V$-tuples of sets.  The map can be coded up by a matrix $ a $ as in~\eqref{nIJ}, where $ M $ there is now the number of edges in the GDS.

If $ V $ is the number of vertices in a GDS, the $ V $-tuple $(K^1, \dots, K^V) $  of fractal sets generated by the GDS is a very particular $ V$-variable  $ V $-tuple.  If the address $ \boldsymbol{a} = a_0a_1\dots a_k \dots $ for $(K^1, \dots, K^V) $    is as in \eqref{nseqcode}, then $ a_k = a$ for all $ k  $. 
Unlike the situation for $ V $-variable fractals as discussed in Remark~\ref{rm32}, there are at most $ V $ distinct subtrees which can be obtained from the tree codes $ \omega^v $  for $ K^v $ \emph{regardless} of the level of the initial node of the subtree.  

More generally, if  $(K^1, \dots, K^V) $ is generated by a GDMS then  for each $ k $,  $  a_{k+1} $ is determined just by $   a_k $ and by the incidence matrix for the GDMS.  Each subtree $ \omega^v\rfloor \sigma  $ is completely determined   by the value $ \omega^v(\sigma) \in \Lambda$ at its base node $ \sigma $ and by the ``branch'' $ \sigma_k $ in $ \sigma = \sigma_1\dots \sigma_k $.

Realisations of random fractals generated by a random GDS are almost surely not   $V $-variable, and are more akin to standard random fractals as in Definition~\ref{def211}.  One comes closer to   $ V $-variable fractal sets by introducing a notion of a  homogeneous random GDS   fractal set analogous to that of a homogeneous random fractal as in Remark~\ref{hft}.   But then one does not obtain a class of $ V $-variable fractals together with its associated probability distribution unless one makes the same definitions as in Section~\ref{secvvtc}.  This would be quite unnatural in the setting of GDS fractals, for example it would require one edge from any vertex to any vertex.
\end{remark}

\section{Approximation Results as $ V  \to \infty $.}\label{secvti}

Theorems~\ref{rc} and \ref{thVapprox} enable one to obtain empirical samples of standard random fractals up to any prescribed degree of approximation by using sufficiently large $ V $ in Theorem~\ref{th41}(iv).  This is useful since  even single realisations of  random fractals are computationally expensive to generate by standard methods.  Note that although the matrices used to compute samples of $ V $-variable fractals  are typically  of order  $ V\times V  $, they are sparse with bandwidth~$ M $.

The next theorem improves the exponent in \cite{BHS05}*{Theorem 12} and removes the dependence on~$  M $.  The difference comes from using the third rather than the first version of \eqref{mk}  in the proof.

\begin{theorem} \label{rc}  If $ d_{MK} $ is the Monge-Kantorovitch metric  then  $ d_{MK}(\rho_V,\rho_\infty) \leq  1.4\, V^{-1/3}$. 
\end{theorem}

\begin{proof}
We construct random tree codes $  W_V  $  and  $ W_\infty  $ with $ \dist  W_V = \rho_V $ and $ \dist W_\infty = \rho_\infty $. In order to apply the last equality in 
\eqref{mk} we want the expected distance between $  W_V  $ and $ W_\infty $, determined by their joint distribution, to be as small as possible.

Suppose $ \boldsymbol{A} = A_0A_1A_2 \ldots $  is a random address with  $ \dist \boldsymbol{A} = P_V^\infty $.  
Let $ W_V = \omega^{\boldsymbol{A}}_1(\sigma ) $ be the corresponding random tree code, using the notation of  \eqref{wI} and \eqref{tcc}.  It follows from Definition~\ref{df38} that $ \dist  W_V = \rho_V $.

Let the random integer $  K= K(\boldsymbol{A})$  be the greatest integer such that, for $ 0 \leq j \leq K $,  if 
$ | \sigma | = | \sigma' | = j $ and $ \sigma \neq \sigma' $ then $ \widehat{J}^{\boldsymbol{A} }_1(\sigma) \neq 
\widehat{J}^{\boldsymbol{A} }_1(\sigma') $ in \eqref{tcc}. 
Thus with $ v=1  $ as in Example~\ref{rem17} the nodes of $  T $  are placed in distinct buffers up to and including level $ K$. 

Let $  W_\infty $  be any random tree code such that
\begin{align*}
& \text{if } | \sigma | \leq K \text{ then } W_\infty(\sigma )  =  W_V( \sigma),   \\
& \text{if } | \sigma | > K \text{ then } \dist W_\infty(\sigma ) = P \text{ and }  
  W_\infty(\sigma ) \text{ is independent of }   W_\infty(\sigma' ) \text{ for all } \sigma'\neq \sigma.
  \end{align*}
It follows from the definition of $ K $ that $W_\infty(\sigma )  $ are iid with distribution $  P $  \emph{for all} $ \sigma $ and so \mbox{$ \dist W_\infty =\rho_\infty $}. 

For any $  k $ and for  $ V \geq M^k  $,
{\allowdisplaybreaks
\begin{align*}
\expected d&(W_V,W_\infty)\\ 
&= \expected ( d(W_V,W_\infty) \mid K\geq k) \cdot  \prob(K\geq k) +
   \expected ( d(W_V,W_\infty) \mid K<k)\cdot   \prob(K<k)  \\
&\leq \frac{1}{M^{k+1}} +   \frac{1}{M} \prob(K<k)\quad  
\text{by \eqref{eqtrees} and since $ W_V(\emptyset) = W_\infty (\emptyset) $}         \\
&=\frac{1}{M^{k+1}} +  \frac{1}{M}
\left(
1- \prod_{i=1}^{M-1} \left( 1 - \frac{i}{V}\right) 
           \prod_{i=1}^{M^2-1} \left( 1 - \frac{i}{V}\right) \cdot
          \,  \dots\, \cdot  
            \prod_{i=1}^{M^k-1}\left( 1 - \frac{i}{V}\right) \right)\\
&\leq \frac{1}{M^{k+1}} +   \frac{1}{MV} \left( \sum_{i=1}^{M-1} i 
         + \sum_{i=1}^{M^2-1} i +\dots +  \sum_{i=1}^{M^k-1} i \right) 
                     \intertext{\qquad \qquad \qquad \qquad \qquad since $ \Pi_{i=1}^n(1-a_i) \geq 1 - \sum_{i=1}^n a_i $  for $ a_i \geq 0 $}
&\leq \frac{1}{M^{k+1}} +   \frac{1}{2MV} \left( M^2 + M^4 +\dots +M^{2k}   \right) \\
& \leq \frac{1}{M^{k+1}} + \frac{M^{2(k+1)}}{2MV(M^2-1)} 
\leq \frac{1}{M^{k+1}} + \frac{2M^{2k-1}}{3V},
   \end{align*}}
assuming $ M\geq 2 $ for the last inequality.  The estimate is trivially true if $ M=1 $ or $ V < M^k $. 

Choose  $    x $ so $ M^x = \left(\frac{3V}{4}\right)^{1/3}  $, this being the value of $ x $ which minimises  $   \frac{1}{M^{x+1}} + \frac{2M^{2x-1}}{3V} $.  Choose  $  k $ so $ k \leq x < k+1 $.  Hence from \eqref{mk}
\begin{align*} 
d_{MK}(\rho_V,\rho_\infty)
&\leq \expected d(W_V,W_\infty)  
\leq \left(\frac{3V}{4}\right)^{-1/3}  \! + \frac{2}{3MV}\left( \frac{3V}{4} \right)^{2/3}\\
&\leq V^{-1/3} \left( \left(\frac{3}{4}\right)^{-1/3} \! +
                    \frac{1}{3} \left( \frac{3}{4} \right)^{2/3}   \right)
                    \leq 1.37 V^{-1/3}.  \tag*{$  \square $}  
\end{align*}
\end{proof}  

\begin{remark}[\emph{No analogous estimate for $ d_P $ is possible in Theorem~\ref{rc}}]\label{nodp}
The support of $ \rho_V$  converges to the support of $ \rho_\infty $ in the Hausdorff metric by Theorem~\ref{prop33}.  However, $ \rho_V \nrightarrow \rho_\infty $  in the $ d_P $ metric as $ V \to \infty $.

To see this suppose $ M\geq 2 $, fix $ j,k\in \{1,\dots,M \} $ with $ j \neq k $ and let 
$ 
E = \{ \omega \in \Omega : \omega(j) = \omega (k) \},
$ 
where $ j$ and $ k  $ are interpreted as sequences of length one in $  T $.
According to the probability distribution $ \rho_\infty  $, $ \omega (j ) $ and $ \omega (k) $ are independent if $ j \neq k $.  For the probability distribution $ \rho^*_V $  there is  a positive  probability $ 1/V $  that $ J(1,j) = J(1,k) $, in which case $ \omega_1 (j) = \omega_1 (k)$ must be equal from Proposition~\ref{adct}, while if $ J(1,j) \neq J(1,k) $ then $ \omega_1(j) $ and $ \omega_1 (k) $ are independent.
 Identifying the  probability distribution  $ \rho_V $ on $ \Omega $ with the projection of $ \rho^*_V $ on $ \Omega^V $  in the first coordinate direction it follows $ \rho_\infty(E) <  \rho_V(E) $.
 
 However, $ d(\omega',E) \geq 1/M $ if $ \omega' \notin E  $  since in this case for $ \omega \in E $  either $ \omega'(j)\neq \omega(j) $  or $ \omega'(k)\neq \omega(k) $.  Hence for $ \epsilon < 1/M $, $  E^\epsilon= E $    and so 
$ \rho_\infty(E^\epsilon) = \rho_\infty(E ) <   \rho_V(E)$.  It follows that 
 $d _P(\rho_V, \rho_\infty)  \geq  1/M$ for all~$  V $ if $ M \geq 2 $.
   \end{remark} 

\begin{theorem}\label{thVapprox}
Under the assumptions of Theorem~\ref{th41},
\begin{equation} \label{eq44a}
\begin{aligned} 
d_{\mathcal{H}}(\mathcal{K}_V, \mathcal{K}_\infty), \ 
d_{\mathcal{H}}(\mathcal{M}_V, \mathcal{M}_\infty) &< \frac{2L}{1-r} V^{-{\alpha}}, \\
d _{MK}(\mathfrak{K}_V, \mathfrak{K}_\infty),\
d _{MK}(\mathfrak{M}_V, \mathfrak{M}_\infty)&<\frac{2.8\, L}{1-r} V^{-\widehat{\alpha}}, 
\end{aligned} 
\end{equation}
where 
\[ 
L =  \sup_\lambda \max_m  d(f^\lambda_m(a),a)   , \quad 
   \alpha  = \frac{\log (1/r)}{\log M} , \quad 
  \widehat{\alpha}  =  \alpha /3 \text{ if }  \alpha \leq 1,\
   \widehat{\alpha}  =  1/3 \text{ if  } \alpha \geq 1 .  
    \]
\end{theorem}

\begin{proof} The first two estimates follow from Theorem~\ref{prop33}  and Proposition~\ref{prpdkw}.

\smallskip For the third, let $ W_V $ and $ W_\infty $ be the random codes from the proof of Theorem~\ref{rc}.  In particular, 
\begin{equation} \label{eq521}
 \rho_V = \dist W_V, \quad \rho_\infty = \dist W_\infty, \quad \expected d( W_V, W_\infty) \leq 1.4 V^{-1/3}.
\end{equation}

Let $ \widehat{\Pi} $ be the projection map $ \Omega \mapsto K^\omega $ given by  \eqref{dKo} and \eqref{kpi}. Then
$  \mathfrak{K}_V = \dist \widehat{\Pi} \circ W_V$,  $ \mathfrak{K}_\infty = \dist \widehat{\Pi} \circ W_\infty   $, 
and from the last condition in~\eqref{mk}
\begin{equation} \label{eq522}
d _{MK}(\mathfrak{K}_V, \mathfrak{K}_\infty) \leq \expected 
                d_{\mathcal{H}}(\widehat{\Pi} \circ W_V, \widehat{\Pi} \circ W_\infty).
                \end{equation} 
                
If $ \alpha \leq 1 $ then from   \eqref{dkw} on taking expectations of both sides, using H\"{o}lder's inequality and applying~\eqref{eq521},
\[  
 \expected  d_{\mathcal{H}}(\widehat{\Pi} \circ W_V, \widehat{\Pi} \circ W_\infty)
      \leq \frac{2L}{1-r} \expected d^\alpha (W_V, W_\infty) \leq \frac{2.8\, L}{1-r} V^{-\alpha/3}.
      \]           
                
If $ \alpha \geq 1 $ then from the last two lines in the proof of Proposition~\eqref{prpdkw},
\[             d_{ \mathcal{H} }(K^\omega,K^{\omega'}) \leq \frac{2L}{1-r}\,  d^\alpha (\omega,\omega') ,
\]
and so, using this and  arguing as  before,
\[  
 \expected  d_{\mathcal{H}}(\widehat{\Pi} \circ W_V, \widehat{\Pi} \circ W_\infty)
    \leq \frac{2.8 L}{1-r} V^{-1/3}.
      \]                  

This gives the third estimate.       The fourth estimate is proved in an analogous manner.
\end{proof}         
                             
Sharper estimates can be obtained arguing directly as in the proof of Theorem~\ref{rc}.  In particular, the exponent $ \widehat{\alpha} $ can be replaced by $ \dfrac{\log(1/r)}{\log(M^2/r)}$.


\section{Example  of $ 2 $-Variable Fractals}\label{2vf}

Consider the  family $ \boldsymbol{F}=\{\mathbb{R}^2, U,D, \frac{1}{2},\frac{1}{2} \} $ consisting of two IFSs  $ U=(f_1,f_2) $ (Up with a reflection) and $ D=(g_1,g_2) $ (Down) acting on $ \mathbb{R}^2$, where
\begin{align*}
f_1(x,y) &= \Big( \frac{x}{2} +\frac{3y}{8}- \frac{1}{16},
        \phantom{+}\frac{x}{2}-\frac{3y}{8}+\frac{9}{16}\Big),
&f_2(x,y)  & =        \Big( \frac{x}{2} -\frac{3y}{8}+ \frac{9}{16},
         -\frac{x}{2}-\frac{3y}{8}+\frac{17}{16}\Big),\\
g_1(x,y) &= \Big( \frac{x}{2} +\frac{3y}{8}- \frac{1}{16},
         -\frac{x}{2}+\frac{3y}{8}+\frac{7}{16}\Big),
&g_2(x,y)  & =        \Big( \frac{x}{2} -\frac{3y}{8}+ \frac{9}{16},
         \phantom{+}\frac{x}{2}+\frac{3y}{8}-\frac{1}{16}\Big).
    \end{align*}
The corresponding fractal attractors of $ U $ and $ F $ are shown at the beginning of  Figure  4.   The probability of choice of $ U $ and $ D $ is $ \frac{1}{2}  $ in each case. 

\begin{figure}[htbp]  
   \centering
   \includegraphics[]{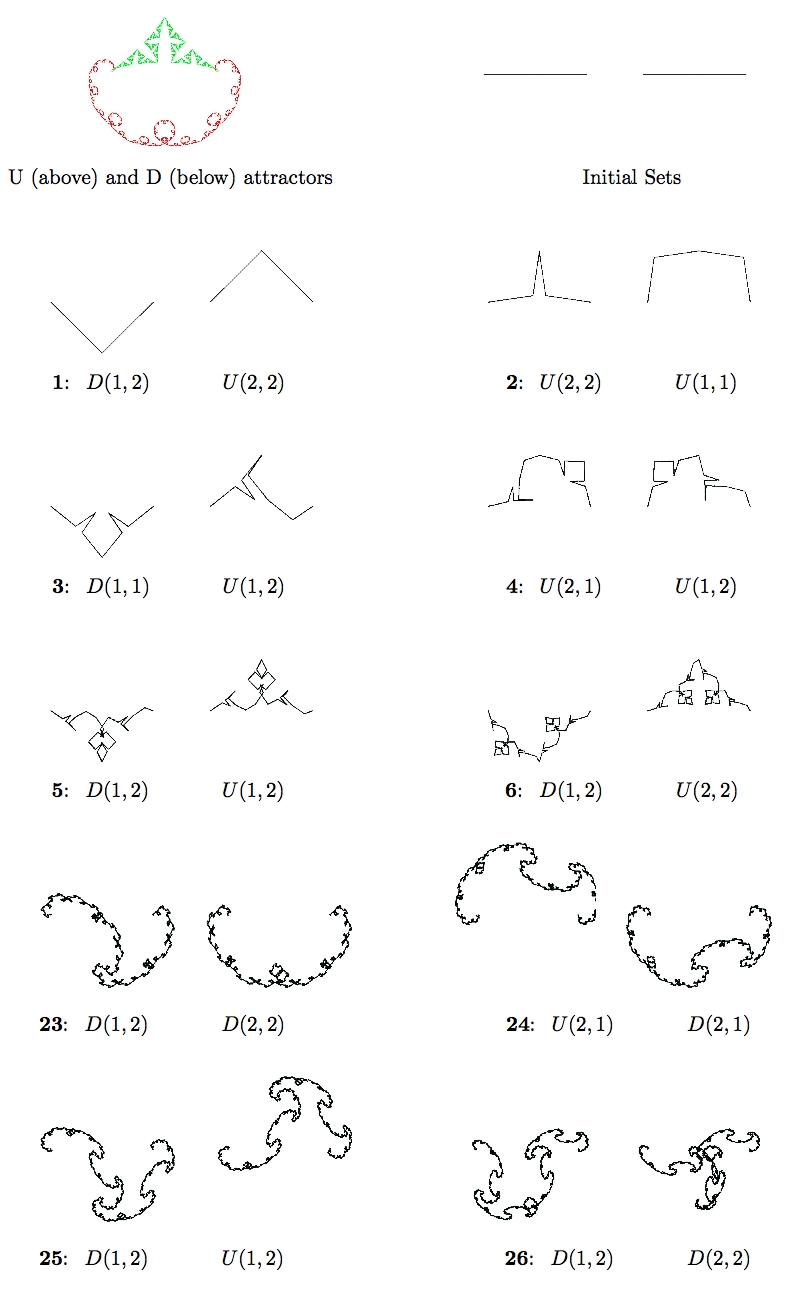} 
   \caption{Sampling 2-variable fractals.}
   \label{fig:example2}
\end{figure}

The  2-variable superIFS acting on pairs of compact sets is 
$ 
  \mathfrak{F}_2^{\mathcal{C} } = 
  (\mathcal{C}(\mathbb{R}^2)^2, \mathcal{F}^a, a\in \mathcal{A}_2,P_2).
$ 
There are  64 maps $ a \in \mathcal{A}_2 $, each a $ 2\times 3 $ matrix.  The probability distribution $ P_2 $ assigns probability $ \frac{1}{64} $  to each $ a\in \mathcal{A}_2 $.     

In iteration 4 in Figure 4 the matrix is
$ 
 a 
    = \begin{bmatrix} 
 			U & 2 & 1 \\
   			U & 1 & 2
	\end{bmatrix}$. 
Applying $ \mathcal{F}^a $ to the pair of sets $ (E_1,E_2) $ from iteration 3 gives
\[ \mathcal{F}^a(E_1,E_2)=
\big(U(E_2,E_1), U(E_1,E_2)\big) = \big(f_1(E_2)\cup f_2(E_1) ,   f_1(E_1)\cup f_2(E_2)  \big).     
\]

The process in Figure 4 begins with a pair of line segments. The first 6 iterations  and iterations 23--26 are shown.  
After about 12 iterations  the sets are independent of the initial sets up to screen resolution.  After this the pairs of sets can be considered as examples of 2-variable 2-tuples of fractal sets corresponding to $ \boldsymbol{F} $.

The generation process gives an MCMC algorithm or ``chaos game''  and acts on the infinite state space $   (\mathcal{C}(\mathbb{R}^2)^2, d_{\mathcal{H} })  $ of pairs of compact sets with the $ d_{\mathcal{H} } $ metric.
The empirical distribution along a.e.\  trajectory from any starting pair of sets   converges weakly to the 2-variable superfractal probability distribution on 2-variable 2-tuples of fractal sets corresponding to $ \boldsymbol{F}  $.  The empirical distribution of first (and second) components  converges weakly to the corresponding natural  probability distribution on 2-variable  fractal sets corresponding to $ \boldsymbol{F}  $.  

\section{Concluding Comments}\label{secext}

\begin{remark}[\emph{Extensions}]\label{hcf}
The restriction that  each IFS $ F^\lambda $ in~\eqref{dfF1} has the same number $ M $ of functions is for notational simplicity only.  

In Definition~\ref{df38}  the independence conditions on the $I $ and $ J $  may be relaxed.  In some modelling situations it would be natural to have a degree of local dependance between $ (I(v),J(v)) $  and $ (I(w),J(w)) $ for $ v $ ``near'' $w $.

The probability distribution $ \rho_V$    is in some sense
the most natural probability distribution on the set of $ V$-variable code
trees since it is inherited from the probability distribution $ P_V^\infty$ 
with the simplest possible probabilistic structure.
We may construct more general distributions on the set of $ V$-variable code
trees by letting $ P_V^\infty$  be non-Bernoulli,  e.g.\ stationary.

Instead of beginning with a family of IFSs in~\eqref{dfF2} one could begin with a family of graph directed IFSs and obtain in this manner the corresponding class of $ V $-variable graph directed fractals. 
\end{remark}

\begin{remark}[\emph{Dimensions}]\label{remar}
 Suppose $\boldsymbol{F} =  \{\mathbb{R}^n, F^\lambda; \lambda \in \Lambda, P \} $   is a family of IFSs satisfying the strong uniform open set condition and whose maps are similitudes.  In a forthcoming paper we compute the a.s.\ dimension of the associated family of $ V $-variable random fractals.
 The idea is to associate to each $ a \in \mathcal{A}_V $ a certain $ V\times V$ matrix and then use the Furstenberg Kesten  theory for products of random matrices to compute a type of pressure function.   
 \end{remark}

\begin{remark}[\emph{Motivation for the construction of $ V $-variable fractals}]\label{mfc}
The original motivation  was to find a chaos game type algorithm for generating  collections of fractal sets whose empirical distributions approximated the probability distribution of standard random fractals.

More precisely, suppose $ \boldsymbol{F} = \{(X,d), F^ \lambda , \lambda \in \Lambda, P \}   $ is a family of IFSs as in~\eqref{dfF2}. Let $ V $ be a large positive integer and $ \mathcal{S} $ be a  collection of  $ V $ compact subsets of $ X$, such that the empirical distribution of $ \mathcal{S} $ approximates the  distribution $ \mathfrak{K}_\infty $  of the standard random fractal  associated to $ \boldsymbol{F} $ by   Definition~\ref{def211}.  Suppose $ \mathcal{S}^* $ is a second collection of $ V $ compact subsets of $ X $ obtained from $ \boldsymbol{F} $ and $ \mathcal{S} $ as follows.  For each $ v \in \{ 1,\dots , V \} $ and  independently  of other $ w \in  \{ 1,\dots , V \} $, select $ E_1, \dots E_M $ from $ \mathcal{S} $  according to the uniform distribution independently with replacement,  and independently  of this select $F^\lambda $ from $ \boldsymbol{F} $ according to  $ P $. Let the $ v $th set in $ \mathcal{S}^* $ be $ F^\lambda (E_1, \dots E_M)  = \bigcup_{1\leq m \leq M} f^\lambda_m(E_m) $.   Then one expects the empirical distribution of $ \mathcal{S}^* $ to also approximate   $ \mathfrak{K}_\infty $.  

The random operator constructed in this manner for passing from $ \mathcal{S} $ to $ \mathcal{S}^* $ is essentially the random operator $ \mathcal{F}^a  $ in Definition~\ref{df42} with $ a \in \mathcal{A}_V $ chosen according to $ P_V $.
\end{remark}

\begin{remark}[\emph{A hierarchy of fractals}]\label{hft}
See Figure~\ref{Table}. 

If $ M=1 $ in \eqref{dfF1} then each $ F^\lambda $ is a trivial IFS $ (f^\lambda) $ containing just one map, and the family $ \boldsymbol{F} $ in 
\eqref{dfF2} can be interpreted as a standard IFS.  If moreover $ V=1 $ then the corresponding superIFS in Definition~\ref{df42} can be interpreted as a standard IFS operating on $ (X,d) $ with set and measure attractors $ K $ and $ \mu $, essentially by identifying singleton subsets of $ X $ with  
elements in $ X $.  For $ M=1 $ and $ V >1 $ the superIFS can be identified with an IFS operating on $ X^V $ with   set and measure attractors $ K_V^* $ and $ \mu_V^* $.  Conversely, any standard IFS can be extended to  a superIFS in this manner.
The projection of $ K_V^* $ in any coordinate direction is $ K $, but 
$ K_V^* \neq K^V $.  
The attractors $K_V^* $ and   $ \mu_V^* $ are called correlated fractals.  The measure $ \mu^* $ provides information on a certain   ``correlation''  between subsets of $ K $.  This  provides a new tool for studying  the structure of standard IFS fractals as we show in a forthcoming paper.
 
The case $ V = 1 $ corresponds to homogeneous random fractals and  has been studied in \cites{Hambly92, Kif, Stenflo01}.   
 The case $ V\to \infty $ corresponds to standard random fractals as defined in Definition~\ref{def211},  see also Section~\ref {secvti}.  See also \cite{Asa} for some graphical examples.
 
For a given class  $ \boldsymbol{F} $ of IFSs and   positive integer $ V>1 $, one obtains a new class of fractals each with the prescribed degree $ V $ of  self similarity at 
every scale.  The associated superIFS provides a rapid way of generating a  sample from this class of $ V $-variable fractals whose empirical distribution approximates the natural probability distribution on the class.

Large $ V $ provides a method for generating a class of correctly distributed approximations  to standard random fractals.   Small $ V $ provides a class of fractals with useful modelling properties.   
  \end{remark}


\begin{theindex}

\index \textbf{Sets and associated metrics}
\item $(X,d)  $:   complete separable metric space, page \pageref{xd}
\item $A^\epsilon $: closed $ \epsilon $ neighbourhood of $ A $, Defn \ref{mx}
\item $ \mathcal{C} $, $ \mathcal{B} \mathcal{C}$:  nonempty compact subsets, nonempty bounded closed subsets, Defn \ref{cx} 
\item $  d_{\mathcal{H}}  $: Hausdorff metric, Defn \ref{cx}

\indexspace
\item \textbf{Measures and associated metrics}
\item $ \mathcal{M}$: unit mass Borel regular (i.e.\ probability) measures, Defn~\ref{mx}
\item $ \mathcal{M}_1$: finite first moment measures in $ \mathcal{M}$, Defn~\ref{mx2}
\item $ \mathcal{M}_b$: bounded support measures in $ \mathcal{M}$, Defn~\ref{mp}
\item $ \mathcal{M}_c$: compact support measures in $ \mathcal{M}$, Defn~\ref{mp}
\item  $ \rho  $: Prokhorov metric, Defn~\ref{mx} 
\item   $ d_{MK} $:  Monge-Kantorovitch  metric, Defn ~\ref{mk} 
\item   $ d_P $:  strong Prokhorov metric, Defn \ref{mp}

\indexspace
\item \textbf{Iterated Function Systems [IFS]}
\item $F = (X, f_\theta,  \theta \in \Theta, W) $: generic IFS, Defn~\ref{def22}
\item $ F = (X,f_1,\dots, f_M, w_1,\dots, w_M) $: generic IFS with a finite number of functions, Defn~\ref{def22}
\item $ F(\cdot)$: $  F $  acting on a set, measure or continuous function (transfer operator), Defn~\ref{def22} and eqn~\eqref{fphi}, also eqns~\eqref{fnu},  \eqref{nuphi} 
\item $ Z^x_n $:  $ n $th (random) iterate  in the  Markov chain given by $ F $ 
and starting from $ x $, eqn~\eqref{28}
\item $ Z^\nu_n $: the corresponding Markov chain with initial distribution $ \nu $, eqn~\eqref{fnu} 
\item $ \widehat{Z}^x_n $: backward  (``convergent'') process, eqn~\eqref{28new}

\indexspace
\item\textbf{Tree Codes and (Standard) Random Fractals}
\item $ \{X, F^\lambda, \lambda  \in\Lambda, P\}   $: a family of IFSs $F^\lambda= (X,f^\lambda_1,\dots,f^\lambda_M,w^\lambda_1,\dots,w^\lambda_M ) $ with   probability distribution $  P $, Defn~\ref{dfF2} 
\item $  T $: canonical $ M $-branching tree, Defn~\ref{dftrees}
\item $|\cdot| $: $ |\sigma| $ is the length of $ \sigma \in T $, Defn~\ref{dftrees} 
\item $ (\Omega,d) $: metric space of tree codes, Defn~\ref{dftrees} 
\item $ \omega \rfloor \tau $:   subtree code of $\omega $ with base node $ \tau $,  Defn~\ref{dftrees}  
\item $ \omega \lfloor k $:   subtree finite code of $ \omega $  of height $  k $, Defn~\ref{dftrees}  
\item $ \rho_\infty $: prob distn on $ \Omega $ induced from $ P $, Defn~\ref{dftrees}  
 
\item $ K^\omega $: (realisation of random) fractal set, Defn~\ref{def28}  
\item $ \mu^\omega $: (realisation of random) fractal, Defn~\ref{def28}  
\item $ K^\omega_\sigma $,  $ \mu^\omega_\sigma $:   subfractals of $ K^\omega $,  $ \mu^\omega $, Defn~\ref{def28} 
\item $ \mathcal{K}_\infty  $, $ \mathcal{M}_\infty  $: collection  of random fractals sets or  measures corresponding to a family of IFSs, eqn~\eqref{eq29}  
\item $ \mathfrak{K}_\infty $,  $ \mathfrak{M}_\infty $: probability distribution on  $ \mathcal{K}_\infty, $, $ \mathcal{M}_\infty  $ induced from  $ \rho_\infty $, Defn~\ref{def211}
\item $ \mathbb{K} $, $ \mathbb{M}  $: random set or measure with  distribution $ \mathfrak{K}_\infty $ or  $ \mathfrak{M}_\infty $, Defn~\ref{def211}
\item $ F(\cdot,\dots,\cdot) $:  IFS $ F $ acting on an M tuple of sets or measures, eqn~\eqref{eq21} 

\indexspace
\item\textbf{$ V $-Variable Tree Codes}
\item $ \Omega_V  $: set of $ V $-variable tree codes, Defn~\ref{Vvar}
\item $ \mathcal{K}_V $: collection of $ V $-variable sets, Defn~\ref{Vvar}   
\item  $ \mathcal{M}_V $: collection of $ V $-variable   measures, Defn~\ref{Vvar}   
\item $ (\Omega^V,d) $: metric space of $ V $-tuples of tree codes, Defn~\ref{omv} 
\item $ \Omega_V^* $: set of $ V $-variable $ V $-tuples of tree codes, Defn~\ref{df34}; attractor of $\boldsymbol{\Phi}_V$, Thm~\ref{char2}
\item $ * $: concatenation operator, Notn~\ref{concat}
\item $\boldsymbol{\Phi}_V$: the IFS $ (\Omega^V, \Phi^a, a \in \mathcal{A}_V )$,  Defn~\ref{dfIFStr}, Defn~\ref{df38}
\item $ \Phi^a$: maps in $\boldsymbol{\Phi}_V $,  Defn~\ref{df38}
\item $ a =(I,J)   \in \mathcal{A}_V $: indices and index set for the maps $ \Phi^a$, Defn~\ref{df38}
\item $I$:  the map $ I:\{1,\dots,V\} \to \Lambda$, Defn~\ref{df38}
\item $J$:  map $ J:\{1,\dots,V\} \times \{1,\dots,M\}\to \{1,\dots,V\}$, $ I:\{1,\dots,V\} \to \Lambda$, Defn~\ref{df38}
\item $ \boldsymbol{a}  $: address  $ a_0a_1\dots a_k \dots \in \mathcal{A}_V^\infty $, Defn~\ref{df38} 
\item $\boldsymbol{\omega}^{\boldsymbol{a} }  =
(\omega_1^{\boldsymbol{a} } ,\dots,\omega_V^{\boldsymbol{a} } )$; $  V$-variable $ V $-tuple of tree codes corresponding to address $ \boldsymbol{a} $, Defn~\ref{dfseqcode} 
\item $ \widehat{J}^{\boldsymbol{a}}_v$:  skeleton map from $ T\to \{1,\dots,V\} $, Defn~\ref{dfskel}
\item $ P_V $, $ P_V^\infty $: prob distributions on $ \mathcal{A}_V $, $ \mathcal{A}^\infty_V $, Defn~\ref{df38}
\item $ \rho_V $, $ \rho_V^* $; prob distributions on $ \Omega_V  $, $ \Omega_V^* $, Defn~\ref{df38}

\indexspace
\item\textbf{$ V $-Variable   Sets and Measures}
\item $ d_{\mathcal{H}}(\cdot,\ldots,\cdot)$: metric  on $ \mathcal{C} (X)^V$,  Defn~\ref{dfextmet}, 
\item $ d_P(\cdot,\ldots,\cdot)$: metric  on  $ \mathcal{M}_c(X)^V $,  Defn~\ref{dfextmet} 
\item $  d_{MK}(\cdot,\ldots,\cdot)$: metric  on  $ \mathcal{M}_1(X)^V $,    Defn~\ref{dfmk}
\item  $ \mathfrak{F}_V^{\mathcal{C}}$: IFS 
$ \left(\mathcal{C} (X)^V,\mathcal{F}^a, a \in \mathcal{A}_V , P_V\right)$, Defn~\ref{concat}
 \item $ \mathfrak{F}_V^{\mathcal{M}_c} $:  IFS $  \left(\mathcal{M}_c (X)^V,\mathcal{F}^a, a \in \mathcal{A}_V , P_V\right)$, Defn~\ref{concat}
 \item $ \mathfrak{F}_V^{\mathcal{M}_1} $:  IFS $  \left(\mathcal{M}_1 (X)^V,\mathcal{F}^a, a \in \mathcal{A}_V , P_V\right)$, Defn~\ref{concat}
\item $ \mathcal{F}^a (\cdot,\dots,\cdot)$: map on $ V $-tuples of sets or measures,  Defn~\ref{concat}  
\item $ \mathcal{K}_V$: collection   of  $V $-variable   fractals sets, Defn~\ref{df34}; projection of $ \mathcal{K}_V^*$ in any coord direction, Def~\ref{df311} 
\item $ \mathcal{K}_V^*$: collection   of  $V $-variable $ V$-tuples of  fractals sets, Defn~\ref{df34}; set attractor of $ \mathfrak{F}_V^{\mathcal{C}}$, Thm~\ref{th41}
\item  $ \mathfrak{K}_V $, $  \mathfrak{K}^*_V $: projected measures on  $ \mathcal{K}_V$ \& $ \mathcal{K}_V^*$ from measure on tree codes, Defn~\ref{df311}; measure attractor of $ \mathfrak{F}_V^{\mathcal{M}_c} $ and its projection, Thm~\ref{th41}
\item $ \mathcal{M}_V$: collection   of  $V $-variable   fractals measures, Defn~\ref{df34}; projection of $ \mathcal{M}_V^*$ in any coord direction, Def~\ref{df311} 
\item $ \mathcal{M}_V^*$: collection   of  $V $-variable $ V$-tuples of  fractals measures, Defn~\ref{df34}; set attractor of $ \mathfrak{F}_V^{\mathcal{M}_c}$, Thm~\ref{th41}
\item  $ \mathfrak{M}_V $: projected measure  on  $ \mathcal{M}_V$ from measure on tree codes, projection of $  \mathfrak{M}^*_V $ in any coord direction, Defn~\ref{df311}
\item $  \mathfrak{M}^*_V $: projected measures on    $ \mathcal{M}_V^*$ from measure on tree codes, Defn~\ref{df311}; measure attractor of $ \mathfrak{F}_V^{\mathcal{M}_c} $ or $ \mathfrak{F}_V^{\mathcal{M}_1} $, Thm~\ref{th41} \& Thm~\ref{th42} 
\end{theindex}

\end{document}